\documentclass[11pt,reqno]{amsart}
\usepackage{latexsym,amssymb,amsfonts,mathrsfs}
\usepackage{amsthm}
\usepackage{amsmath}
\usepackage{a4wide}
\usepackage{enumerate}
\usepackage{comment}
\usepackage[latin1, utf8]{inputenc}
\usepackage{eufrak}
\usepackage[usenames]{color}
\usepackage{graphicx}
\usepackage{lmodern,xfrac,xcolor}
\usepackage[noadjust]{cite}
\usepackage[unicode]{hyperref}

\usepackage{mathtools}
\usepackage[toc]{appendix}
\mathtoolsset{showonlyrefs}
\usepackage{scalerel}

\def\e{\mathrm{e}}
\def\ve{\varepsilon}
\def\me{\mathsf{e}}
\def\mv{\mathsf{v}}

\def\ea{\EuFrak{a}}

\def\ve{\varepsilon}
\def\la{\lambda}

\def\kF{\kappa_{\scaleto{\widetilde{F}}{4.5pt}}}
\def\kG{\kappa_{\scaleto{G}{3.3pt}}}

\def\comp{\mathbb{C}}
\def\real{\mathbb{R}}
\def\nat{\mathbb{N}}
\def\mcA{\mathcal{A}}
\def\mcB{\mathcal{B}}

\def\mcV{\mathcal{V}}
\def\mcE{\mathcal{E}}
\def\mcF{\mathcal{F}}
\def\mcG{\mathcal{G}}
\def\mcH{\mathcal{H}}
\def\mcT{\mathcal{T}}
\def\mcS{\mathcal{S}}
\def\mcX{\mathcal{X}}

\def\mcV{\mathcal{V}}
\def\mcW{\mathcal{W}}

\def\dx{\, dx}
\def\dt{\, dt}
\def\ds{\, ds}

\DeclareMathOperator{\sgn}{sgn}

\newtheorem{theo}{Theorem}
\newtheorem{lemma}[theo]{Lemma}
\newtheorem{prop}[theo]{Proposition}
\newtheorem{cor}[theo]{Corollary}
\newtheorem{defi}[theo]{Definition}
\newtheorem{asum}[theo]{Assumption}
\newtheorem{assum}[theo]{Assumptions}

\theoremstyle{remark}
\newtheorem{rem}[theo]{Remark}

\numberwithin{equation}{section} \numberwithin{theo}{section}

\begin{document}

\title[Stochastic reaction-diffusion equations on networks]{Stochastic reaction-diffusion equations on networks}

\author{M.~Kov\'acs}
\address{Faculty of Information Technology and Bionics\\P\'azm\'any P\'eter Catholic University\\
Budapest, Hungary\\
Chalmers University of Technology and University of Gothenburg\\
Gothenburg, Sweden}
\email{mihaly@chalmers.se}
\thanks{}

\author{E.~Sikolya}
\address{Department of Applied Analysis and Computational Mathematics\\
E\"otv\"os Lor\'and University\\
Budapest, Hungary\\
Alfr\'ed R\'enyi Institute of Mathematics,\\ 
Budapest, Hungary}
\email{eszter.sikolya@ttk.elte.hu}
\thanks{M.~Kov\'acs was supported by Marsden Fund of the Royal Society of New Zealand grant no. 18-UOO-143, VR grant number 2017-04274 and NKFI grant number K-131501.}

\date{\today}


\begin{abstract}
We consider stochastic reaction-diffusion equations on a finite network represented by a finite graph.
On each edge in the graph a multiplicative cylindrical Gaussian noise driven reaction-diffusion equation is given supplemented by a dynamic Kirchhoff-type law perturbed by multiplicative scalar Gaussian noise in the vertices. The reaction term on each edge is assumed to be an odd degree polynomial, not necessarily of the same degree on each edge, with possibly stochastic coefficients and negative leading term. We utilize the semigroup approach for stochastic evolution equations in Banach spaces to obtain existence and uniqueness of solutions with sample paths in the space of continuous functions on the graph. In order to do so we generalize existing results on abstract stochastic reaction-diffusion equations in Banach spaces.
\end{abstract}

\maketitle

%%%%%%%%%%%%%%%%%%%%%%%%%%%%%%%%%%%%%%%%%%%%%%%%%%%%%%%%%%%%%%%%%%%%%%%%%%%
\section{Introduction}
%%%%%%%%%%%%%%%%%%%%%%%%%%%%%%%%%%%%%%%%%%%%%%%%%%%%%%%%%%%%%%%%%%%%%%%%%%%%

We consider a finite connected network, represented by a
finite graph $\mathsf{G}$ with $m$ edges $\me _1,\dots,\me _m$ and $n$
vertices $\mv_1,\dots,\mv_n$. We normalize and parametrize the edges on the interval $[0,1]$. We denote by $\Gamma(\mv_i)$ the set of all the indices of the edges
having an endpoint at $\mv _i$, i.e.,
\[\Gamma(\mv _i)\coloneqq\left\{j\in \{1,\ldots,m\}: \me _j(0)=\mv _i\hbox{ or } \me _j(1)=\mv _i\right\}.\]
We denote by $\Phi\coloneqq(\phi_{ij})_{n\times m}$
 the so-called incidence matrix of the graph $\mathsf{G}$, see Subsection \ref{subsec:systeq} for more details.
For $T>0$ given we consider the stochastic system written formally as
\begin{equation}\label{eq:stochnet}
\left\{\begin{array}{rcll}
\dot{u}_j(t,x)&=& (c_j u_j')'(t,x)+d_j(x)\cdot u_j'(t,x)&\\
&-& p_j(x)u_j(t,x)+f_j(t,x,u_j(t,x))&\\
&+&h_j(t,x,u_j(t,x))\frac{\partial w_j}{\partial t}(t,x), &t\in(0,T],\; x\in(0,1),\; j=1,\dots,m, \\
u_j(t,\mv _i)&=&u_\ell (t,\mv _i)\eqqcolon r_i(t), &t\in(0,T],\; \forall j,\ell\in \Gamma(\mv _i),\; i=1,\ldots,n,\\
\dot{r}_i(t)&=&[M r(t)]_{i}&\\
&&+\sum_{j=1}^m \phi_{ij}\mu_{j} c_j(\mv_i) u'_j(t,\mv_i)&\\
&&+g_i(t,r_i(t))\dot{\beta}_i(t), &t\in(0,T],\; i=1,\ldots,n,\\
u_j(0,x)&=&\mathsf{u}_{j}(x), &x\in [0,1],\; j=1,\dots,m.\\
r_i(0)&=&\mathsf{r}_{i}, & i=1,\dots ,n.
\end{array}
\right.
\end{equation}
where $(\beta_i(t))_{t\in [0,T]}$ are independent scalar Brownian motions and  $(w_j(t))_{t\in [0,T]}$ are independent cylindrical Wiener-processes defined in the Hilbert space $L^2(0,1; \mu_j dx)$ for some $\mu_j>0$, $j=1,\dots,m$. The reaction terms $f_j$  are assumed to be odd degree polynomials, with possible different degree on different edges, and with possibly stochastic coefficients and negative leading term, see \eqref{eq:fjdef}. The diffusion coefficients $g_i$ and $h_j$ are assumed to be locally Lipschitz continuous and satisfy appropriate growths conditions \eqref{eq:gidefnov} and \eqref{eq:hjdefnov}, respectively, depending on the maximum and minimum degrees of the polynomials $f_j$ on the edges. These become linear growth conditions when the degrees of the polynomials $f_j$ on the different edges coincide. The coefficients of the linear operator satisfy standard smoothness assumptions, see Subsection \ref{subsec:systeq}, while the matrix $M$ satisfies Assumptions \ref{as:M} and $\mu_j$, $ j=1,\dots,m$, are positive constants. 

While deterministic evolution equations on networks are well studied, see,  \cite{Al84,Al86,Al94,ABN01,Be85,Be88,Be88b,BN96,Ca97,CF03,EK19,Ka66,KMS07,KS05,LLS94,Lu80,MS07,Mu07,MR07,Mu14,Ni85} which is, admittedly, a rather incomplete list, the study of their stochastic counterparts is surprisingly scarce despite their strong link to applications. In \cite{BMZ08} additive L\'evy noise is considered that is square integrable with drift being a cubic polynomial. In  \cite{BZ14} multiplicative square integrable L\'evy noise is considered but with globally Lipschitz drift and diffusion coefficients and with a small time dependent perturbation of the linear operator. Paper \cite{BM10} treats the case when the noise is an additive fractional Brownian motion and the drift is zero. In \cite{CP17a} multiplicative Wiener perturbation is considered both on the edges and vertices with globally Lipschitz diffusion coefficient and zero drift and time-delayed boundary condition.  Finally, in \cite{CP17}, the case of multiplicative Wiener noise is treated with bounded and globally Lipschitz continuous drift and diffusion coefficients and noise both on the edges and vertices. 

In all these papers the semigroup approach is utilized in a Hilbert space setting and the only work that treats non-globally Lipschitz continuous coefficients is \cite{BMZ08} but the noise is there is additive and square-integrable. In this case, energy arguments are possible using the additive nature of the equation which does not carry over to the multiplicative case. Therefore, we use an entirely different toolset based on the semigroup approach for stochastic evolution equations in Banach spaces \cite{vNVW08}. For results on classical stochastic reaction-diffusion equations on domains in $\mathbb{R}^n$ we refer, for example, to \cite{BG99,BP99,Ce03,CeF17,Pe95}.  The papers \cite{KvN12,KvN19} introduce a rather general abstract framework for treating such equations using the above mentioned semigroup approach of \cite{vNVW08}. Unfortunately, the framework is still not quite general enough to apply it to \eqref{eq:stochnet}. The reason for this is as follows. One may rewrite \eqref{eq:stochnet} as an abstract stochastic Cauchy problem of the form \eqref{eq:SCP}. The setting of   \cite{KvN12,KvN19}  require a space $B$ which is sandwiched between some UMD Banach space $E$ of type 2 where the operator semigroup $S$ generated by the linear operator $A$ in the equation is strongly continuous and analytic, and the domain of some appropriate fractional power of $A$. The semigroup $S$ is assumed to be strongly continuous on $B$, a property that is used in an essential way via approximation arguments (for example, Yosida approximations). The drift $F$ is assumed to be a map from $B$ to $B$ and assumed to have favourable properties on $B$. In the abstract Cauchy problem \eqref{eq:SCPn} corresponding to \eqref{eq:stochnet}, such a space $\mathcal{B}$ given by \eqref{eq:mcB} plays the role of $B$. Here $\mcB$ is the space of continuous functions on the graph that are also continuous across the vertices (more precisely, isomorphic to it). But then the abstract drift $\mathcal{F}$ given by \eqref{eq:mcFdef} does not map $\mathcal{B}$ to itself unless very unnatural conditions on the coefficients of $f_j$ are introduced. One may consider the larger space $\mathcal{E}^c$ introduced in Definition \ref{defi:mcEc}, where continuity is only required on each edge (but not necessarily across the vertices). Then  $\mathcal{F}$ given by \eqref{eq:mcFdef} maps $\mathcal{E}^c$ to itself and $\mathcal{F}$ still has favourable properties on $\mathcal{E}^c$ and $\mathcal{E}^c$ is still sandwiched the same way as $\mathcal{B}$. The price to pay for considering this larger space is the loss of strong continuity of the semigroup $\mcS$ generated by the linear operator $\mcA$ of \eqref{eq:SCPn} on $\mathcal{E}^c$. However, the semigroup will be analytic on $\mathcal{E}^c$. This property that can be exploited in various approximation arguments that do not require strong continuity, see, for example, \cite{Lun95}.  Such arguments are used in the seminal paper \cite{Ce03}, where a system of reaction-diffusion equations are studied but, unlike in the present work, with a diagonal solution operator, and polynomials with the same degree in each component (see \cite[Remark 5.1, 2.]{Ce03}). While the framework of \cite{Ce03} is less general than that of \cite{KvN12,KvN19} the approximation arguments in the former do not use strong continuity. We therefore prove abstract results (Theorems \ref{theo:KvN4.3gen} and \ref{theo:KvN4.9gen}) concerning existence and uniqueness of the solution of \eqref{eq:SCPn} in the setting of \cite{vNVW08}, similar to that of Theorems 4.3 and 4.9 in  \cite{KvN12,KvN19}, but without the requirement that $\mcS$ is strongly continuous on the sandwiched space and using similar approximation arguments as in  \cite{Ce03}. The assumption on $\mathcal{F}$, in particular, Assumptions \ref{assum:mainvN4.9mod}(5), is also more general than the corresponding assumptions in \cite{Ce03} and  \cite{KvN12,KvN19} so that we may consider polynomials with different degrees on different edges. 

The main results of the paper concerning the system \eqref{eq:stochnet} are contained in Theorems \ref{theo:SCPnsolcont} and \ref{theo:Holderreg}. In Theorem \ref{theo:SCPnsolcont}  we show that there is a unique mild solution of \eqref{eq:stochnet} with values in $\mathcal{E}^c$. While, as we explained above, we cannot work with the space $\mathcal{B}$ directly, in Theorem \ref{theo:Holderreg} we prove via a bootstrapping argument that the solution, in fact, has values in $\mathcal{B}$; that is, the solution is also continuous across the vertices even when the initial condition is not. 

%Theorems \ref{theo:SCPnsolcont} and \ref{theo:Holderreg} contain corresponding results when noise is only present in the vertices with relaxed integrability condition on the initial data.

The paper is organized as follows. In Section \ref{sec:determnetwork} we collect partially known semigroup results for the linear deterministic version of \eqref{eq:stochnet}. For the sake of completeness, while the general approach is known, we include the proof of Proposition \ref{prop:formA} in Appendix \ref{app:pfpropdetermL2}, and the key technical results needed in the proof Proposition \ref{prop:sgrextend}, which is the main result of this section, in Appendix \ref{app:extsgrLp}.
In Section \ref{sec:srde} we prove two abstract results, Theorems \ref{theo:KvN4.3gen} and \ref{theo:KvN4.9gen}, concerning the existence and uniqueness of mild solutions \eqref{eq:SCP}. In Section \ref{sec:srden} we apply the abstract results to \eqref{eq:stochnet}. In order to do so, in Subsection \ref{sec:srdenprep} we first prove various embedding and isometry results and, in Proposition \ref{prop:mcAonmcEc}, we prove that the semigroup $S$ is analytic on $\mathcal{E}^c$. 
Subsection \ref{subsec:mainresults} contains the main existence and uniqueness results concerning  \eqref{eq:stochnet}, see Theorems \ref{theo:SCPnsolcont}, \ref{theo:SCPnsolconthj0} and \ref{theo:Holderreg}, \ref{theo:Holderreghj0}. In the latter cases we treat separately the models where stochastic noise is only present in the nodes.
%Subsection \ref{subsec:mainresults1} contains the existence and uniqueness results concerning  \eqref{eq:stochnet} when noise is only present in the nodes (Theorems \ref{theo:SCPnsolcont} and \ref{theo:Holderreg}) while Subsection \ref{subsec:mainresults1} contains analogous results when noise is also introduced on the edges (Theorems \ref{theo:SCPn2solcont} and \ref{theo:Holderreg2}). 

%%%%%%%%%%%%%%%%%%%%%%%%%%%%%%%%%%%%%%%%%%%%%%%%%%%%%%%%%%%%%%%%%%%%%%%%%%%%%%%%%%%%%%%%%%%%%%%%%%%%%%%%%
\section{The heat equation on a network}\label{sec:determnetwork}
%%%%%%%%%%%%%%%%%%%%%%%%%%%%%%%%%%%%%%%%%%%%%%%%%%%%%%%%%%%%%%%%%%%%%%%%%%%%%%%%%%%%%%%%%%%%%%%%%%%%%%%%%

%%%%%%%%%%%%%%%%%%%%%%%%%%%%%%%%%%%%%%%%%%%%%%%%%%%%%%%%%
\subsection{The system of equations}\label{subsec:systeq}
%%%%%%%%%%%%%%%%%%%%%%%%%%%%%%%%%%%%%%%%%%%%%%%%%%%%%%%%%

We consider a finite connected network, represented by a
finite graph $\mathsf{G}$ with $m$ edges $\me _1,\dots,\me _m$ and $n$ vertices $\mv_1,\dots,\mv_n$. We normalize and parameterize the edges on the interval $[0,1]$. The
structure of the network is given by the $n\times m$ matrices
$\Phi^+\coloneqq (\phi^+_{ij})$ and $\Phi^-\coloneqq (\phi^-_{ij})$ defined by
\begin{equation}\label{eq:fiijpm}
\phi^+_{ij}\coloneqq \left\{
\begin{array}{rl}
1, & \hbox{if } \me _j(0)=\mv _i,\\
0, & \hbox{otherwise},
\end{array}
\right.
\qquad\hbox{and}\qquad
\phi^-_{ij}\coloneqq \left\{
\begin{array}{rl}
1, & \hbox{if } \me _j(1)=\mv _i,\\
0, & \hbox{otherwise,}
\end{array}
\right.
\end{equation}
for $i=1,\ldots ,n$ and $j=1,\ldots m.$ We denote by $\me _j(0)$ and $\me _j(1)$ the $0$ and the $1$ endpoint of the edge $\me _j$, respectively.
We refer to~\cite{KS05} for terminology. The $n\times m$ matrix
$\Phi\coloneqq (\phi_{ij})$ defined by \[\Phi\coloneqq \Phi^+-\Phi^-\] is known in
graph theory as \emph{incidence matrix} of the graph $\mathsf{G}$. Further,
let $\Gamma(\mv_i)$ be the set of all the indices of the edges
having an endpoint at $\mv _i$, i.e.,
\[\Gamma(\mv _i)\coloneqq \left\{j\in \{1,\ldots,m\}\colon  \me _j(0)=\mv _i\hbox{ or } \me _j(1)=\mv _i\right\}.\]
For the sake of simplicity, we will denote the values of a continuous function defined on the (parameterized) edges of the graph, that is of
\[f=\left(f_1,\ldots ,f_m\right)^{\top}\in \left(C[0,1]\right)^m\cong C\left([0,1],\mathbb{C}^m\right)\]
at $0$ or $1$ by $f_j(\mv_i)$ if $\me _j(0)=\mv _i$ or $\me _j(1)=\mv _i$, respectively, and $f_j(\mv_i)\coloneqq 0$ otherwise, for $j=1,\ldots ,m$.

We start with the problem
\begin{equation}\label{netcp}
\left\{\begin{array}{rclll}
\dot{u}_j(t,x)&=& (c_j u_j')'(t,x)-p_j(x)u_j(t,x), &t> 0,\; x\in(0,1),\; j=1,\dots,m, & (a)\\
u_j(t,\mv _i)&=&u_\ell (t,\mv _i)\eqqcolon r_i(t), &t> 0,\; \forall j,\ell\in \Gamma(\mv _i),\; i=1,\ldots,n,& (b)\\
\dot{r}_i(t)&=&[M r(t)]_{i}&&\\
&&+\sum_{j=1}^m \phi_{ij}\mu_{j} c_j(\mv_i) u'_j(t,\mv_i), &t> 0,\; i=1,\ldots,n,& (c)\\
u_j(0,x)&=&\mathsf{u}_{j}(x), &x\in [0,1],\; j=1,\dots,m & (d)\\
r_i(0)&=&\mathsf{r}_{i}, & i=1,\dots ,n. & (e)
\end{array}
\right.
\end{equation}
on the network. Note that $c_j(\cdot)$ and $u_j(t,\cdot)$ are
functions on the edge $\me_j$ of the network, so that the
right-hand side of~$(\ref{netcp}a)$ reads in fact as
\[(c_j u_j')'(t,\cdot)=\frac{\partial}{\partial x}\left( c_j
\frac{\partial}{\partial x}u_j\right)(t,\cdot), \qquad t\geq 0,\; j=1,\ldots,m.\]

The functions $c_1,\ldots,c_m$ are (variable) diffusion coefficients or conductances, and we assume that 
\[0<c_j\in C^1[0,1],\quad j=1,\ldots,m.\]

The functions $p_1,\ldots,p_m$ are nonnegative, continuous functions, hence
\begin{equation}\label{eq:pjpos}
0\leq p_j\in C[0,1],\quad j=1,\ldots,m.
\end{equation}

Equation~$(\ref{netcp}b)$ represents the continuity of the values attained by the system at the vertices in each time instant, and we denote by $r_i(t)$ the common functions values in the vertice $i$, for $i=1,\ldots,n$ and $t\geq 0$.

In $(\ref{netcp}c)$, $M\coloneqq\left(b_{ij}\right)_{n\times n}$ is a matrix satisfying the following
\begin{asum}\label{as:M}
The matrix $M=\left(b_{ij}\right)_{n\times n}$ is 
\begin{enumerate}
	\item real, symmetric,% $M\not\equiv 0$;
	%\item negative definite\marginpar{ez eleg?}
	\item for $i \neq k,$ $b_{ik}\geq 0$, that is, $M$ has positive off-diagonal;
	\item 
\[b_{ii}+\sum_{k\neq i}b_{ik}< 0, \quad i=1,\dots ,n.\]
\end{enumerate}
\end{asum}

On the right-hand-side, $[M r(t)]_{i}$ denotes the $i$th coordinate of the vector $M r(t)$. The coefficients 
\[0<\mu_j,\quad  j=1,\ldots,m\] 
are strictly positive constants that influence the distribution of impulse happening in the ramification nodes according to the Kirchhoff-type law~$(\ref{netcp}c)$.

We now introduce the $n\times m$ \emph{weighted incidence matrices}
\[\Phi^+_w\coloneqq (\omega^+_{ij})\text{ and }\Phi^-_w\coloneqq (\omega^-_{ij})\] 
with entries
\begin{equation}\label{eq:matrixFiw}
\omega^+_{ij}\coloneqq \left\{
\begin{array}{ll}
\mu_j c_j(\mv_i), & \hbox{if } \me _j(0)=\mv _i,\\
0, & \hbox{otherwise},
\end{array}
\right. \qquad\hbox{and}\qquad \omega^-_{ij}\coloneqq \left\{
\begin{array}{ll}
\mu_j c_j(\mv_i), & \hbox{if } \me _j(1)=\mv _i,\\
0, & \hbox{otherwise}.
\end{array}
\right.
\end{equation}
With these notations, the Kirchhoff law $(\ref{netcp}c)$ becomes
\begin{equation}
\label{eq:Kir}
\dot{r}(t)=Mr(t)+\Phi_w^+ u'(t,0)-\Phi_w^- u'(t,1), \qquad t\geq 0.
\end{equation}

In equations (\ref{netcp}d) and (\ref{netcp}e) we pose the initial conditions on the edges and the vertices, respectively.

%%%%%%%%%%%%%%%%%%%%%%%%%%%%%%%%%
\subsection{Spaces and operators}
%%%%%%%%%%%%%%%%%%%%%%%%%%%%%%%%%

We are now in the position to rewrite our system \eqref{netcp} in form of an abstract Cauchy problem, following the concept of \cite{KMS07}. First we consider the (complex) Hilbert space
\begin{equation}\label{eq:E2}
E_2\coloneqq\prod_{j=1}^m L^2(0,1; \mu_j dx)
\end{equation}
as the \emph{state space} of the edges, endowed with the natural inner product
\[\langle u,v\rangle_{E_2}\coloneqq \sum_{j=1}^m \int_0^1 u_j(x)\overline{v_j(x)} \mu_j dx,\qquad
u=\left(\begin{smallmatrix}u_1\\ \vdots\\
u_m\end{smallmatrix}\right),\;v=\left(\begin{smallmatrix}v_1\\ \vdots\\
v_m\end{smallmatrix}\right)\in E_2.\]
Observe that $E_2$ is isomorphic to $\left( L^2(0,1)\right)^m$ with equivalence of norms.

We further need the \emph{boundary space} $\mathbb{C}^{n}$ of the vertices. According to $(\ref{netcp}b)$ we will consider such functions on the edges of the graph those values coincide in each vertex, that is, that are \emph{continuous in the vertices}. Therefore we introduce the \emph{boundary value operator} 
\[L\colon\left(C[0,1]\right)^m\subset E_2\to \mathbb{C}^n\]
with
\begin{align}\label{eq:Ldef}
D(L) & = \left\{u\in \left(C[0,1]\right)^m\colon u_j(\mv _i)=u_\ell (\mv _i),\; \forall j,\ell\in \Gamma(\mv _i),\; i=1,\ldots,n\right\};\notag\\
L u & \coloneqq\left(r_1,\ldots ,r_n\right)^{\top}\in \mathbb{C}^n,\quad r_i=u_j(\mv _i)\text{ for some } j\in \Gamma(\mv _i),\; i=1,\ldots,n.
\end{align}
The condition $u(t,\cdot)\in D(L)$ for each $t>0$ means that $(\ref{netcp}b))$ is for the function $u(\cdot,\cdot)$ satisfied.

On $E_2$ we define the operator
\begin{equation}\label{eq:opAmax}
A_{max}\coloneqq\begin{pmatrix}
\frac{d}{dx}\left(c_1 \frac{d}{dx}\right)-p_1 & & 0\\
 & \ddots &\\
0 & & \frac{d}{dx}\left(c_m \frac{d}{dx}\right)-p_m\\
\end{pmatrix}
\end{equation}
with domain
\begin{equation}\label{eq:domAmax}
D(A_{max})\coloneqq\left(H^2(0,1)\right)^m\cap D(L).
\end{equation}
This operator can be regarded as \emph{maximal} since no other boundary condition except continuity is imposed for the functions in its domain.

We further define the so called \emph{feedback operator} acting on $D(A_{max})$ and having values in the boundary space $\mathbb{C}^{n}$ as
\begin{align}\label{eq:opC}
D(C) & = D(A_{max});\\
C u & \coloneqq \Phi_w^+ u'(0)-\Phi_w^- u'(1).
\end{align}
With these notations, the Kirchhoff law $(\ref{netcp}c)$ becomes
\begin{equation}
\label{eq:KirC}
\dot{r}(t)=Mr(t)+Cu(t), \qquad t\geq 0,
\end{equation}
compare with \eqref{eq:Kir}.

With these notations, we can finally rewrite~\eqref{netcp} in form of an abstract Cauchy problem on the product space of the state space and the boundary space,
\begin{equation}\label{eq:mcE_2def}
\mathcal{E}_2\coloneqq E_2\times \ell^2(\mathbb{C}^{n})
\end{equation}
endowed with the natural inner product
\begin{equation}
\langle U,V \rangle_{\mathcal{E}_2}\coloneqq \langle u,v\rangle_{E_2}+\langle r,q\rangle_{\comp^n}\quad \text{ for }U,V\in \mathcal{E}_2, \quad
U=\left(\begin{smallmatrix}u\\ r\end{smallmatrix}\right), V=\Big(\begin{smallmatrix}v\\ q\end{smallmatrix}\Big),
\end{equation}
where
\[\langle r,q\rangle_{\comp^n}=\sum_{i=1}^n r_i\overline{q_i}\]
is the usual scalar product in $\mathbb{C}^{n}.$\\

We now define the operator matrix $\mathcal{A}_2$ on $\mathcal{E}_2$ as
\begin{equation}
\label{eq:calA}
\mcA_2=\begin{pmatrix}
	A_{max} & 0\\
	C & M
\end{pmatrix}
\end{equation}
with domain
\begin{equation}
\label{eq:domcalA}
D(\mcA_2)=\left\{ \left(\begin{smallmatrix}u\\ r\end{smallmatrix}\right)\in D(A_{max})\times \mathbb{C}^{n}\colon L u=r\right\}.
\end{equation}
We use the notation $\mcA_2$ because the operator will be later extended to other $L^p$-spaces, see Proposition \ref{prop:sgrextend}.

Using this, \eqref{netcp} becomes
\begin{equation}\label{eq:acp}
\left\{\begin{array}{rcll}
\dot{U}(t)&=& \mathcal{A}_2U(t), &t\geq 0,\\
U(0)&=&\left(\begin{smallmatrix}\mathsf{u}\\ \mathsf{r}\end{smallmatrix}\right),
\end{array}
\right.
\end{equation}
with $\mathsf{u}=(\mathsf{u}_1,\dots ,\mathsf{u}_m)^{\top}$, $\mathsf{r}=(\mathsf{r}_1,\dots ,\mathsf{r}_n)^{\top}$.

%%%%%%%%%%%%%%%%%%%%%%%%%%%%%%%%%%%%%%%%%%%%%%%%%%%%%%%%%%
\subsection{Well-posedness of the abstract Cauchy problem}
%%%%%%%%%%%%%%%%%%%%%%%%%%%%%%%%%%%%%%%%%%%%%%%%%%%%%%%%%%

To prove well-posedness of \eqref{eq:acp} we associate a sesquilinear form with the operator $\left(\mathcal{A}_2, D(\mathcal{A}_2)\right)$, similarly as e.g. in \cite{CP17} or (for the case of diagonal $M$) in \cite{MR07} and testify appropriate properties of it.
Define
\begin{align}\label{eq:form}
\ea(U,V)&\coloneqq \sum_{j=1}^m\int_0^1 \mu_j c_j(x) u'_j(x) \overline{g'_j(x)} dx+\sum_{j=1}^m\int_0^1 \mu_j p_j(x)u_j(x) \overline{v_j(x)} dx-\langle Mr,q\rangle_{\comp^n},\\ 
&\text{ for }U=\left(\begin{smallmatrix}u\\ r\end{smallmatrix}\right),\, V=\Big(\begin{smallmatrix}g\\ q\end{smallmatrix}\Big)
\end{align}
on the Hilbert space $\mathcal{E}_2$ with domain
\begin{equation}\label{eq:domform}
D\left(\ea\right)=\mathcal{V}\coloneqq \left\{U=\left(\begin{smallmatrix}u\\ r\end{smallmatrix}\right)\colon u\in \left(H^1(0,1)\right)^m\cap D(L),\, r\in \mathbb{C}^{n},\; L u=r\right\}.
\end{equation}
The next definition can be found e.g. in \cite[Sec.~1.2.3]{Ou05}.
\begin{defi}\label{defi:opassform}
From the form $\ea$ -- using the Riesz representation theorem -- we can obtain a unique operator
$\left(\mcB,D(\mcB)\right)$ in the following way:
\begin{equation*}
\begin{array}{rcl}
D(\mcB)&\coloneqq&\left\{U\in \mcV:\exists V\in \mcE_2 \hbox{ s.t. } \ea(U,H)=\langle V,H\rangle_{\mcE_2}\; \forall H\in \mcV\right\},\\
\mcB U&\coloneqq&-V.
\end{array}\end{equation*}
We say that the operator $\left(\mcB,D(\mcB)\right)$ is \emph{associated with the form $\ea$}.
\end{defi}

%In the following, we will claim that the operator associated to the form $\ea$ is $(\mcA_2,D(\mcA_2)).$ Since the proof is rather technical, we leave it for Appendix \ref{app:pfpropdetermL2}.

\begin{prop}\label{prop:formA}
The operator associated to the form $\ea$ \eqref{eq:form}--\eqref{eq:domform} is $(\mcA_2,D(\mcA_2))$ from \eqref{eq:calA} -- \eqref{eq:domcalA}.
\end{prop}
\begin{proof}
See the proof of Proposition \ref{prop:formAapp}.
\end{proof}

In the subsequent proposition we will prove well-posedness of \eqref{eq:acp} not only on the Hilbert space $\mcE_2$ but also on $L^p$-spaces, which will be crucial for our later results. Therefore we introduce the following notions. Let

\begin{equation}\label{eq:Ep}
E_p\coloneqq\prod_{j=1}^m L^p(0,1; \mu_j dx),\quad p\in[1,\infty]
\end{equation}
and
\begin{equation}\label{eq:mcEp}
\mcE_p\coloneqq E_p\times \ell^p(\comp^n),
\end{equation}
endowed with the norm
\begin{equation}\label{eq:Epnorm}
\|U\|^p_{\mcE_p}\coloneqq \sum_{j=1}^m \|u_j\|^p_{L^p(0,1; \mu_j dx)}+\|r\|^p_{\ell^p},\quad U=\left(\begin{smallmatrix} u\\r\end{smallmatrix}\right)\in \mcE_p,\quad u\in E_p,\, r\in \mathbb{C}^{n},\quad p\in[1,\infty),
\end{equation}
\begin{equation}\label{eq:Einftynorm}
\|U\|_{\mcE_{\infty}}\coloneqq \max\left\{\max_{j=1,\dots ,m} \|u_j\|_{L^{\infty}(0,1)};\max_{i=1,\dots ,n}|r_i|\right\},\quad U=\left(\begin{smallmatrix} u\\r\end{smallmatrix}\right)\in \mcE_{\infty},\quad u\in E_{\infty},\, r\in \mathbb{C}^{n}.
\end{equation}

We now state the main result regarding well-posedness of \eqref{eq:acp}. The proof uses a technical lemma that can be found in Appendix \ref{app:extsgrLp}.

\begin{prop}\label{prop:sgrextend}
Let $M$ satisfy Assumptions \ref{as:M}. 
\begin{enumerate}[1.]
\item The operator $\left(\mcA_2, D(\mcA_2)\right)$ defined in \eqref{eq:calA} -- \eqref{eq:domcalA} is self-adjoint and positive definite. Furthermore, it generates a $C_0$ analytic, contractive, positive one-parameter semigroup $\left(\mcT_2(t)\right)_{t\geq 0}$ on $\mcE_2$.
\item The semigroup $(\mcT_2(t))_{t\geq 0}$ extends to a family of analytic, contractive, positive one-parameter semigroups $(\mcT_p(t))_{t\geq 0}$ on $\mcE_p$ for $1\leq p\leq \infty$, generated by $(\mcA_p,D(\mcA_p))$. These semigroups are strongly continuous if $p\in[1,\infty)$ and consistent in the sense that if $q,p\in [1,\infty]$ and $q\geq p$ then
\begin{equation}\label{eq:sgrconsistent}
\mcT_p(t)U=\mcT_q(t)U\text{ for }U\in \mcE_q.
\end{equation}
\end{enumerate}
\end{prop}
\begin{proof}
To prove part 1 of the claim observe that the form $\ea$ is symmetric since $M$ is real and symmetric, see the proof of \cite[Cor.~3.3]{Mu07}. By mimicking the proofs of \cite[Lem.~3.1]{Mu07} and \cite[Lem.~3.2]{MR07}, we obtain that the form $\ea$ is densely defined, continuous and closed. As a consequence of Assumptions \ref{as:M} we have that $M$ is negative definite. Now using that $c_j>0$, $p_j\geq 0$, $j=1,\dots ,m$, it is straightforward that the form $\ea$ is also accretive. Hence, by \cite[Prop. 1.24., Prop.~1.51, Thm.~1.52]{Ou05} we obtain that $(\mcA_2,D(\mcA_2))$ is densely defined, self-adjoint, positive definite, dissipative and sectorial. As a consequence we obtain that it generates a $C_0$-semigroup $\left(\mcT_2(t)\right)_{t\geq 0}$ on $\mcE_2$ having the properties claimed.\\
Applying Lemma \ref{lem:subMarkov} and the properties of $\mcA_2$ from part 1, we can use \cite[Thm.~1.4.1]{Da90} and $\mcE_q\hookrightarrow \mcE_p$ if $q\geq p$ to obtain all the statements in part 2 but analyticity of the semigroups. To prove this we apply \cite[Thm.~4.1]{PZ11} with the assumption $-M$ positive definite. We can mimick the proof step by step but we have to modify the spaces and operators appropriately to consider them as well as in the vertices. We only point out the crucial definition of the space $W_{bc}$ that should be
\begin{equation*}
W_{bc}\coloneqq\left\{\left(\begin{smallmatrix}\psi\\ r\end{smallmatrix}\right)\in (C^{\infty}(\real)\cap L^{\infty}(\real))\times\real^n\colon \begin{aligned} & |\psi'|_{\infty}\leq 1, |\psi''|_{\infty}\leq 1, \text{ and }\psi\text{ takes a constant}\\
& \text{value }\psi_v\text{ for }x=0,1,\dots ,m;\; r=\{\psi_v\}^n\end{aligned}\right\} \end{equation*}
in our case. Thus we obtain that the semigroup $\mcT_2$ on $\mcE_2$ admits Gaussian upper bound (see also \cite[Thm.~2.13]{CP17}). Since $\mcT_2$ is analytic by part 1, we can apply \cite[$\mathsection$7.4.3]{Ar04} (or \cite[Thm.~5.4]{AterE97}) and obtain that $\mcT_p$ is analytic for each $p\in [1,\infty].$
\end{proof}

Here and in what follows the notion of semigroup and its generator is understood in the sense of \cite[Def.~3.2.5]{ABHN11}. That is a strongly continuous function $T \colon (0,\infty)\to \mathcal{L}(E)$, (where $E$ is a Banach space and $\mathcal{L}(E)$ denotes the bounded linear operators on $E$) satisfying
\begin{enumerate}[(a)]
	\item $T(t + s) = T(t)T(s)$, $s,\, t > 0$,
	\item there exists $c > 0$ such that $\|T(t)\| \leq c$ for all $t\in (0, 1]$,
	\item $T(t)x = 0$ for all $t > 0$ implies $x = 0$
\end{enumerate}
is called a \emph{semigroup}. By \cite[Thm.~3.1.7]{ABHN11} there exist constants $M,\omega\geq 0$ such that $\|T(t)\|\leq M\e^{\omega t}$ for all $t>0$. From \cite[Prop.~3.2.4]{ABHN11} we obtain that there exists an operator $A$ with $(\omega ,\infty ) \subset \rho(A)$ and
\begin{equation}\label{eq:}
R(\la,A)=\int_0^{\infty}\e^{-\la t}T(t)\dt\quad (\la>\omega),
\end{equation}
and we call $(A,D(A))$ the \emph{generator} of $T$. The semigroup is \emph{strongly continuous} or $C_0$, that is $T \colon [0,\infty)\to \mathcal{L}(E)$ and
\begin{align}
T(t + s) &= T(t)T(s),\quad  s,\, t \geq 0,\\
T(0)&=Id
\end{align}
if and only if its generator is densely defined, see \cite[Cor.~3.3.11.]{ABHN11}. According to \cite[Def.~3.7.1]{ABHN11}, we call the semigroup \emph{analytic} or \emph{holomorphic} if there exists $\theta \in (0,\frac{\pi}{2}]$ such that $T$ has a holomorphic extension to
\[\Sigma_{\theta}\coloneqq \{z\in\comp\setminus \{0\}\colon |\mathrm{arg}\, z|<\theta\}\]
which is bounded on $\Sigma_{\theta'}\cap \{z\in\comp\colon |z|\leq 1\}$ for all $\theta'\in (0,\theta)$. We say that an analytic semigroup is \emph{contractive} when the semigroup operators considered on the positive real half-axis are contractions.
%\end{rem}

We also can prove -- analogously as in \cite[Lem.~5.7]{MR07} -- that the generators $(\mcA_p,D(\mcA_p))$ for $p<\infty$ have in fact the same form as in $\mcE_2$, with appropriate domain.
\begin{lemma}\label{lem:mcAp}
For all $p\in [1,\infty)$ the semigroup generators $(\mcA_p,D(\mcA_p))$ are given by the operator defined in~\eqref{eq:calA} with
domain
\begin{equation}\label{eq:mcAp}
D(\mcA_p)=\left\{\left(\begin{smallmatrix} u\\r\end{smallmatrix}\right)\in \left(\prod _{j=1}^m W^{2,p}(0,1;\mu_j dx)\cap D(L)\right)\times \comp^n\colon L u=r \right\}.
\end{equation}
%In particular, $\mcA_p$ has compact resolvent for $p\in [1,\infty)$.
\end{lemma}
As a summary we obtain the following theorem.
\begin{theo}\label{theo:determLp}
The first order problem~\eqref{netcp} considered with $\mcA_p$ instead of $\mcA_2$ is well-posed on $\mcE_p$, $p\in[1,\infty)$, i.e., for all initial data
$\left(\begin{smallmatrix} \mathsf{u}\\\mathsf{r} \end{smallmatrix}\right)\in \mcE_p$ the
problem~\eqref{netcp} admits a unique mild solution that continuously depends on the initial data.
\end{theo}

%%%%%%%%%%%%%%%%%%%%%%%%%%%%%%%%%%%%%%%%%%%%%%%%%%%%%%%%%%%%%%%%%%%%%%%%%%%%%%%%%%%%%%%%%%%%%%%%%%%%%%%%%%%%%%%%%%%%%%
\section{Abstract results for a stochastic reaction-diffusion equation}\label{sec:srde}
%%%%%%%%%%%%%%%%%%%%%%%%%%%%%%%%%%%%%%%%%%%%%%%%%%%%%%%%%%%%%%%%%%%%%%%%%%%%%%%%%%%%%%%%%%%%%%%%%%%%%%%%%%%%%%%%%%%%%%
Let $(\Omega,\mathscr{F},\mathbb{P})$ is a complete probability space endowed with a right continuous filtration $\mathbb{F}=(\mathscr{F}_t)_{t\in [0,T]}$ for a given $T>0$. Let $(W_H(t))_{t\in [0,T]}$ be a cylindrical Wiener process, defined on $(\Omega,\mathscr{F},\mathbb{P})$, in some Hilbert space $H$ with respect to the filtration $\mathbb{F}$; that is,
  $(W_H(t))_{t\in [0,T]}$ is $(\mathscr{F}_t)_{t\in [0,T]}$-adapted and for all $t>s$, $W_H(t)-W_H(s)$ is independent of $\mathscr{F}_s$.

First we prove a generalized version of the result of M.~Kunze and J.~van Neerven, concerning the following abstract equation
\begin{equation}\tag{SCP}\label{eq:SCP}
\left\{
\begin{aligned}
d X(t)&=[A X(t)+F(t,X(t))+\widetilde{F}(t,X(t))]dt+G(t,X(t))d W_{H}(t)\\
X(0)&=\xi,
\end{aligned}
\right.
\end{equation}
see \cite[Sec.~3]{KvN12}. 

In what follows let $E$ be a real Banach space. Occasionally -- without being stressed -- we have to pass to appropriate complexification (see e.g. \cite{MST99}) when we use sectoriality arguments.
If we assume that $(A, D(A))$ generates a strongly continuous, analytic semigroup $S$ on the Banach space $E$ with $\Vert S(t)\Vert\leq M\e^{\omega t}$, $t\geq 0$ for some $M\geq 1$ and $\omega\in\real$, then for $\omega'>\omega$ the fractional powers $(\omega'-A)^{\alpha}$ are well-defined for all $\alpha\in(0,1).$ In particular, the fractional domain spaces
\begin{equation}\label{eq:fractdom}
E^{\alpha}\coloneqq D((\omega'-A)^{\alpha}),\quad \|u\|_{\alpha}\coloneqq \|(\omega'-A)^{\alpha}u\|,\quad u\in D((\omega'-A)^{\alpha})
\end{equation}
are Banach spaces. It is well-known (see e.g. \cite[$\mathsection$II.4--5.]{EN00}), that up to equivalent norms, these space are independent of the choice of $w'.$

For $\alpha\in(0,1)$ we define the extrapolation spaces $E^{-\alpha}$ as the completion of $E$ under the norms $\|u\|_{-\alpha}\coloneqq \|(\omega'-A)^{-\alpha}u\|$, $u\in E$. These spaces are independent of $\omega'>\omega$ up to an equivalent norm.

We fix $E^0\coloneqq E$.

\begin{rem}\label{rem:omega0}
If $\omega=0$ (hence, the semigroup $S$ is bounded), then by \cite[Proposition 3.1.7]{Haase06}
we can choose $\omega'=0$. That is,
\[E^{\alpha}\cong D((-A)^{\alpha}),\quad \alpha\in [0,1),\]
when $D((-A)^\alpha)$ is equipped with the graph norm.
\end{rem}
Let $E^c$ be a Banach space, $\|\cdot \|$ will denote $\|\cdot\|_{E^c}$. For $u\in E^c$ we define the \emph{subdifferential of the norm at} $u$ as the set
\begin{equation}\label{eq:subdiff}
\partial\|u\|\coloneqq \left\{u^*\in \left(E^c\right)^*\colon \|u^*\|_{\left(E^c\right)^*}=1\text{ and }\langle u,u^*\rangle=1\right\}
\end{equation}
which is not empty by the Hahn-Banach theorem.

We introduce the following assumptions for the operators in (SCP). 

\begin{assum}$ $\label{assum:mainvN4.3mod}
\begin{enumerate}
	\item Let $E$ be a UMD Banach space of type $2$ and $(A, D(A))$ a densely defined, closed and sectorial operator on $E$.
	\item We have continuous (but not necessarily dense) embeddings for some $\theta\in (0,1)$ \[E^{\theta}\hookrightarrow E^c\hookrightarrow E.\]
	\item The strongly continuous analytic semigroup $S$ generated by $(A, D(A))$ on $E$ restricts to an analytic, contractive semigroup, denoted by $S^{c}$ on $E^c$, with generator $(A^c,D(A^c))$.
	\item The map $F\colon [0,T]\times\Omega\times E^c\to E^c$ is continuous in the first variable and locally Lipschitz continuous in the third variable in the sense that for all $r>0$, there exists a constant $L_{F}^{(r)}$ such that
	      \[\left\|F(t,\omega,u)-F(t,\omega,v)\right\|\leq L_{F}^{(r)}\|u-v\|\]
	      for all $\|u\|,\|v\|\leq r$ and $(t,\omega)\in [0,T]\times \Omega$ and there exists a constant $C_{F,0}\geq 0$ such that
				\[\left\|F(t,\omega,0)\right\|\leq C_{F,0},\quad t\in[0,T],\; \omega\in\Omega.\]
				Moreover, for all $u\in E^c$ the map $(t,\omega)\mapsto F(t,\omega,u)$ is strongly measurable and adapted.\\
				Finally, for suitable constants $a',b'\geq 0$ and $N\geq 1$ we have
				\[\langle A u+F (t,u+v), u^*\rangle\leq a'(1+\|v\|)^N+b'\|u\|\]
				for all $u\in D(A|_{E^c})$, $v\in E^c$ and $u^*\in\partial\|u\|,$ see \eqref{eq:subdiff}.
	\item For some constant $\kF\geq 0$, the map $\widetilde{F}\colon [0,T]\times\Omega\times E^c\to E^{-\kF}$ is globally Lipschitz continuous in the third variable, uniformly with respect to the first and second variables.
				Moreover, for all $u\in E^c$ the map $(t,\omega)\mapsto \widetilde{F}(t,\omega,u)$ is strongly measurable and adapted.\\
				Finally, for some $d'\geq 0$ we have 
				\begin{equation}\label{eq:Ftildegrow}
				\left\|\widetilde{F}(t,\omega,u)\right\|_{E^{-\kF}}\leq d'\left(1+\|u\|\right)
				\end{equation}
				for all $(t,\omega,u)\in [0,T]\times \Omega\times E^c.$
	\item Let $\gamma(H,E^{-\kG})$ denote the space of $\gamma$-radonifying operators from $H$ to $E^{-\kG}$ for some constant $\kG\geq 0$, see e.g. \cite[Sec.~3.1]{KvN12}.
	      Then the map $G\colon [0,T]\times\Omega\times E^c\to \gamma(H,E^{-\kG})$ is locally Lipschitz continuous in the sense that for all $r>0$, there exists a constant $L_{G}^{(r)}$ such that
	      \[\left\|G(t,\omega,u)-G(t,\omega,v)\right\|_{\gamma(H,E^{-\kG})}\leq L_{G}^{(r)}\|u-v\|\]
	      for all $\|u\|,\|v\|\leq r$ and $(t,\omega)\in [0,T]\times \Omega$.
				Moreover, for all $u\in E^c$ and $h\in H$ the map $(t,\omega)\mapsto G(t,\omega,u)h$ is strongly measurable and adapted.\\
				Finally, for some $c'\geq 0$ we have 
				\begin{equation}\label{eq:ggrow}
				\left\|G(t,\omega,u)\right\|_{\gamma(H,E^{-\kG})}\leq c'\left(1+\|u\|\right)^{\frac{1}{N}}
				\end{equation}
				for all $(t,\omega,u)\in [0,T]\times \Omega\times E^c.$
\end{enumerate}
\end{assum}

For a thorough discussion of UMD Banach spaces we refer to \cite{Bu01}. Banach spaces of type $p\in[1,2]$ are treated in depth in \cite[Sec.~6]{AK16}. In particular, any $L^p$-space with $p\in[2,\infty)$ has type $2$. However, the space of continuous functions on any locally compact Hausdorff space is not a UMD space.

\begin{rem}
Assumptions \ref{assum:mainvN4.3mod}(1)--(4) and (6) are -- in the first 3 cases slightly modified versions of -- Assumptions (A1), (A5), (A4), (F') and (G') in \cite{KvN12}. Assumption \ref{assum:mainvN4.3mod}(5) is the assumption of \cite[Prop.~3.8]{KvN12} on $\widetilde{F}$. The main difference is that here the semigroup $S$ is not necessarily strongly continuous on $E^c$ but is analytic and that the embedding of $E^{\theta}\hookrightarrow E^c$ is not necessarily dense. 

Instead of \eqref{eq:ggrow} in Assumption \ref{assum:mainvN4.3mod}(6) one may assume a slightly improved estimate
\[
\left\|G(t,\omega,u)\right\|_{\gamma(H,E^{-\kG})}\leq c'\left(1+\|u\|\right)^{\frac{1}{N}+\ve}
\]
for some small $\ve>0$ depending on the parameters, as it is stated in (G') of \cite{KvN12}. For simplicity we chose not to include the small $\ve$ explicitly because to prove our main results it will not be needed. 

We use that $E$ is of type $2$ in a crucial way e.g.~in the first step of the proof of Theorem  \ref{theo:KvN4.3gen}, \eqref{eq:estimateGn} and \eqref{eq:L2embedding}, obtaining that the simple Lipschitz and growth conditions for the operator $G$ in Assumption \ref{assum:mainvN4.3mod}(6) suffice, see \cite[Lem.~5.2]{vNVW08}.
\end{rem}

\begin{rem}\label{rem:analdisscontr}
In Assumptions \ref{assum:mainvN4.3mod}(3) we use the fact that since $S$ is analytic on $E$ and by Assumptions \ref{assum:mainvN4.3mod}(2), $D(A)\subset E^{\theta}\hookrightarrow E^c$ holds, $S$ leaves $E^c$ invariant. Hence, the restriction $S^c$ of $S$ on $E^c$ makes sense, and by assumption, $S^c$ is an analytic contraction semigroup on $E^c$. Using \cite[Prop.~3.7.16]{ABHN11} we obtain that this is equivalent to the fact that the generator $A^c$ of $S^c$ is sectorial and dissipative. Note that since $S^c$ is not necessarily strongly continuous, $A^c$ is not necessarily densely defined.

However, one can easily prove that $(A^c,D(A^c))$ is the \emph{part} of $(A, D(A))$ in $E^c$. By the definition of the generator (see \cite[Def.~3.2.5]{ABHN11}) we have that for $\lambda>0$ and $u\in E^c$
\begin{equation}
R(\la,A^c)u =\int_0^{\infty}\e^{-\la t}S^c(t)u\dt=\int_0^{\infty}\e^{-\la t}S (t)u\dt.
\end{equation}
By Assumptions \ref{assum:mainvN4.3mod}(2), the last integral also converges in the norm of $E$. Thus,
\begin{equation}%\label{eq:}
R(\la,A^c)u=R(\la,A)u,\quad \lambda>0,\; u\in E^c.
\end{equation}
The inclusion $D(A)\subset E^c$ implies that in this case also $R(\la,A)u\in E^c$ is satisfied. Hence, we conclude that
\[R(\la,A^c)u=R(\la,A|_{E^c})u,\quad \lambda>0,\; u\in E^c.\] 
Since $\la>0$ and $u\in E^c$ were arbitrary, the equality $(A^c,D(A^c))=(A|_{E^c}, D(A|_{E^c}))$ holds.
\end{rem} 

Recall that a mild solution of \eqref{eq:SCP} is a solution of the following integral equation
\begin{align}\label{eq:mildsol}
X(t)&=S(t)\xi+\int_0^tS(t-s)\left(F(s,X(s))+\widetilde{F}(s,X(s))\right)\ds+\int_0^tS(t-s)G(s,X(s))\,dW_H(s)\notag\\
&\eqqcolon S(t)\xi+S\ast F(\cdot,X(\cdot))(t)+S\ast \widetilde{F}(\cdot,X(\cdot))(t)+S\diamond G(\cdot,X(\cdot))(t)
\end{align}
where
\[S\ast f(t)=\int_0^tS(t-s)f(s)\ds\]
denotes the "usual" convolution, and
\[S\diamond g(t)=\int_0^tS(t-s)g(s)\,dW_{H}(s)\]
denotes the stochastic convolution with respect to $W_{H}.$ We also implicitly assume that all the terms on the right hand side of  \eqref{eq:mildsol} are well-defined.

The following result is analogous to the statement of \cite[Lem.~4.4]{KvN12} but with the semigroup $S^c$ being an analytic contraction semigroup on $E^c$ which is not necessarily strongly continuous. The main difference in the proof is the use of a different approximation argument as the one in \cite{KvN12} uses the strong continuity of $S^c$ on $E^c$ (the latter denoted by $B$ there) in a crucial manner.

\begin{lemma}\label{lem:44}
Let $S^c$ be an analytic contraction semigroup  on $E^c$. Let $x\in E^c$ and $F\colon [0,T]\times\Omega\times E^c\to E^c$ satisfy condition (4) of Assumptions \ref{assum:mainvN4.3mod}, and denote by $C\coloneqq\max\{a',b'\}$. Assume that $u\in C((0,T];E^c)\cap L^{\infty}(0,T;E^c)$ and $v\in C([0,T]\;E^c)$ satisfy
\begin{equation}\label{eq:lem44mildsol}
u(t)=S^c(t)x+\int_0^tS^c(t-s)F(s,u(s)+v(s))\ds,\quad \forall t\in[0,T].
\end{equation}
Then
\begin{equation}\label{eq:lem44}
\left\|u(t)\right\|\leq \e^{C\cdot t}\left(\|x\|+\int_0^t C\left(1+\|v(s)\|\right)^N \ds\right),\quad \forall t\in[0,T].
\end{equation}
\end{lemma}
\begin{proof}
Let $A^c$ be the generator of $S^c$, that is a sectorial and dissipative operator (see Remark \ref{rem:analdisscontr}) and fix $v\in C([0,T]\;E^c)$ satisfying \eqref{eq:lem44mildsol}. Thus we can use methods from the proofs of \cite[Prop.~6.2.2]{Ce01b} and \cite[Lem.~5.4]{Ce03}. Taking $\la\in\rho(A^c)$, we introduce the problem on $E^c$
\begin{equation}\label{eq:pflem44la}
\dot{z}(t)=A ^cz(t)+F(t,z(t)+v(t)),\; t>a,\quad z(a)=\la R(\la,A^c)x\eqqcolon x_{\la}.
\end{equation}
First take $a=0$. Since $x_{\la}\in D(A^c)$ and $A^c$ is sectorial, by \cite[Thm.~7.1.3(i)]{Lun95} this problem has a unique local mild solution in $C([0,\delta],E^c))$ for some $\delta>0$, called $u_{\la,\delta}$, satisfying
\begin{equation}\label{eq:pflem44uladelta}
u_{\la,\delta}(t)=S^c(t)x_{\la}+\int_0^tS^c(t-s)F(s,u_{\la,\delta}(s)+v(s))\ds
\end{equation}
for $t\in[0,\delta]$

We now set $a=\delta$ and take $u_{\la,\delta}(\delta)$ instead of $x_{\la}$ in \eqref{eq:pflem44la}. Since $u_{\la,\delta}(\delta)$ satisfies \eqref{eq:pflem44uladelta} with $t=\delta$ and $S^c$ is analytic, we obtain that $u_{\la,\delta}(\delta)$ belongs to $\overline{D(A^c)}$. Hence, by \cite[Thm.~7.1.3(i)]{Lun95}, there exists $\ve>0$ and a unique local mild solution of \eqref{eq:pflem44la} in $C([\delta,\delta+\ve],E^c)$, called $u_{\la,\ve}$, satisfying
\begin{equation}\label{eq:pflem44ula2}
u_{\la,\ve}(t)=S^c(t)u_{\la,\delta}(\delta)+\int_{\delta}^tS^c(t-s)F(s,u_{\la,\ve}(s)+v(s))\ds,\quad t\in [\delta,\delta+\ve].
\end{equation} Defining
\[u_{\la,\alpha}\coloneqq
\begin{cases}
u_{\la,\delta}(t), & t\in [0,\delta],\\
u_{\la,\ve}(t), & t\in (\delta,\delta+\ve]
\end{cases}
\]
yields a solution $u_{\la,\alpha}\in C([0,\alpha],E^c)$ of \eqref{eq:pflem44la} with $a=0$. Again, $u_{\la,\alpha}(\alpha)$ can be taken as initial value for problem \eqref{eq:pflem44la} with $a=\alpha$, and the above procedure may be repeated indefinitely, up to construct a noncontinuable solution defined in a maximal time interval $I(x_{\la})$.
As in \cite[Def.~7.1.7]{Lun95} we define by $I(x_{\la})$ as the union of all the intervals $[0,\alpha]$ such that \eqref{eq:pflem44la} has a mild solution $u_{\la,\alpha}$ on this interval belonging to $C([0,\alpha],E^c)$. Denote by 
\[\tau(x_{\la})\coloneqq \sup I(x_{\la})\] 
and
\begin{equation}\label{eq:solula}
u_{\la,\max}(t)\coloneqq u_{\la,\alpha}(t), \text{ if }t\in [0,\alpha]\subset I(x_{\la})
\end{equation}
which is well defined thanks to the uniqueness part of \cite[Thm.~7.1.3(i)]{Lun95}.

In the following we first show that the desired norm estimate \eqref{eq:lem44} holds for the maximal solution $u_{\la,\max}$ on $I(x_{\la})$. At the end we will be able to prove that $I(x_{\la})=[0,T]$.

Fix now $t\in I(x_{\la})$. Then by definition, there exists $\alpha>0$ such that $t\in [0,\alpha]$ and $u_{\la,\max}(t)= u_{\la,\alpha}(t)$ holds for the mild solution $u_{\la,\alpha}\in C([0,\alpha],E^c)$ of \eqref{eq:pflem44la}. For the sake of simplicity, we denote $u_{\la}\coloneqq u_{\la,\alpha}$.

Rewriting \eqref{eq:pflem44uladelta} for $u_{\la}$ we obtain that for $t\in [0,\alpha]$ 
\begin{equation}\label{eq:pflem44ula}
u_{\la}(t)=S^c(t)x_{\la}+\int_0^tS^c(t-s)F(s,u_{\la}(s)+v(s))\ds \eqqcolon S^c(t)x_{\la}+\int_0^tS^c(t-s)f_{\la}(s)\ds,
\end{equation}
where $f_{\la}(s)=F(s,u_{\la}(s)+v(s)).$  Since $f_{\la}(\cdot)\in C([0,\alpha],E^c)$ and by definition, $u_{\la}(0)=x_{\la}\in D(A^c)$ holds, we can apply \cite[Prop.~4.1.8]{Lun95} and obtain that $u_{\la}$ is a strong solution of \eqref{eq:pflem44la} in the sense of \cite[Def.~4.1.1]{Lun95}. This means, that there exists a sequence $(u_{\la,n})\subset C^1([0,\alpha],E^c)\cap C([0,\alpha],D(A^c))$ such that
\begin{equation}\label{eq:lem44approx}
u_{\la,n}\to u_{\la},\quad \dot{u}_{\la,n}-A^cu_{\la,n}\eqqcolon f_{\la,n}\to f_{\la}\text{ in }C([0,\alpha],E^c)
\end{equation}
holds, as $n$ goes to infinity. 
Using that $u_{\la,n}\in C^1([0,\alpha],E^c)$, we have by \cite[Thm.~17.9]{HeStr75} that
\[\frac{d}{dt}\|u_{\la,n}(t)\|=\langle \dot{u}_{\la,n}(t),u_{\la,n}(t)^*\rangle,\text{ for all }u_{\la,n}(t)^*\in\partial \|u_{\la,n}(t)\|.\]
Hence, for all $u_{\la,n}(t)^*\in\partial \|u_{\la,n}(t)\|$,
\begin{align}
\frac{d}{dt}\|u_{\la,n}(t)\|&=\langle A^c u_{\la,n}(t)+f_{\la,n}(t),u_{\la,n}(t)^*\rangle\\
&=\langle A^c u_{\la,n}(t)+F(t,u_{\la,n}(t)+v(t)),u_{\la,n}(t)^*\rangle\\
&+\langle f_{\la}(t)-F(t,u_{\la,n}(t)+v(t)),u_{\la,n}(t)^*\rangle+\langle f_{\la,n}(t)-f_{\la}(t),u_{\la,n}(t)^*\rangle
\end{align}
Using the assumption on $F$, we obtain that
\begin{align}\label{eq:lem44ulanest}
\frac{d}{dt}\|u_{\la,n}(t)\|&\leq a'\left(1+\|v(t)\|\right)^N+b'\|u_{\la,n}(t)\|\notag\\
&+\|f_{\la}-F(\cdot,u_{\la,n}+v)\|_{C([0,t],E^c)}+\|f_{\la,n}-f_{\la}\|_{C([0,t],E^c)}\notag\\
&\leq C\cdot \|u_{\la,n}(t)\|+C\cdot\left(1+\|v(t)\| \right)^N+\ve_{\la,n}
\end{align}
with 
\begin{equation}\label{eq:lem44epsilon}
\ve_{\la,n}=\|f_{\la}-F(\cdot,u_{\la,n}+v)\|_{C([0,t],E^c)}+\|f_{\la,n}-f_{\la}\|_{C([0,t],E^c)}.
\end{equation}
By Gronwall's lemma from \eqref{eq:lem44ulanest} we have
\begin{equation}\label{eq:lem44ulanestGr}
\|u_{\la,n}(t)\|\leq \e^{C\cdot t}\cdot \|u_{\la,n}(0)\|+C\int_0^t\e^{C\cdot (t-s)}\left(1+\|v(s)\|\right)^N\ds+\e^{C\cdot t}\cdot t \cdot \ve_{\la,n}.
\end{equation}
Observe that by the continuity of $F$ and \eqref{eq:lem44approx}, we have that 
\begin{equation}\label{eq:lem44epsilonto0}
\ve_{\la,n}\to 0,\quad n\to\infty.
\end{equation}
Hence, letting $n\to\infty$ in \eqref{eq:lem44ulanestGr} and using \eqref{eq:lem44approx} we obtain
\begin{equation}%\label{eq:}
\|u_{\la}(t)\|\leq \e^{C\cdot t}\cdot \|x_{\la}\|+C\cdot \e^{C\cdot t}\int_0^t\left(1+\|v(s)\|\right)^N\ds.
\end{equation} 
Since $A^c$ generates a contraction semigroup, $\|x_{\la}\|\leq \|x\|$ holds, and we obtain that
\begin{equation}\label{eq:pflem44est}
\|u_{\la}(t)\|\leq \e^{C\cdot t}\cdot\left( \|x\|+\int_0^tC\cdot \left(1+\|v(s)\|\right)^N\ds\right), \qquad t\in [0,\alpha],
\end{equation}
hence
\begin{equation}\label{eq:pflem44estalphaT}
\left\|u_{\la,\alpha}(t)\right\| \leq \e^{C\cdot T}\left(\|x\|+\int_0^T C\cdot\left(1+\|v(s)\|\right)^N \ds\right),
\quad t\in [0,\alpha].
\end{equation}
Since $t\in I(x_{\la})$ was arbitrary, and the right-hand-side of \eqref{eq:pflem44estalphaT} does not depend on $t$, we have
\[\lim\sup_{t\to \tau(x_{\la})}\|u_{\la,\max}(t)\|<\infty.\]
Using \cite[Prop.~7.1.8]{Lun95} (and its corollary) it follows that the solution $u_{\la,\max}$ is also global, hence $I(x_{\la})=[0,T]$. Thus, \eqref{eq:lem44} holds for $u_{\la,\max}$ and $t\in [0,T]$.

Finally, arguing as in the last part of the proof of \cite[Prop.~6.2.2]{Ce01b}, we obtain \eqref{eq:lem44} for $u(t).$
\end{proof}

Following \cite{Ce03}, for a fixed $T>0$ and $q\geq 1$, we define the space
\begin{equation}\label{eq:V_T}
V_{T,q}\coloneqq L^q\left(\Omega;C((0,T];E^c)\cap L^{\infty}(0,T;E^c)\right)
\end{equation}
being a Banach space with norm
\begin{equation}\label{eq:V_Tnorm}
\left\|u\right\|^q_{V_{T,q}}\coloneqq \mathbb{E}\sup_{t\in [0,T]}\|u(t)\|^q,\quad u\in V_{T,q}.
\end{equation} 
This Banach space will play a crucial role for the solutions of \eqref{eq:SCP}.

\begin{theo}\label{theo:KvN4.3gen}
Let $T>0$, let $2<q<\infty$ and suppose that Assumptions \ref{assum:mainvN4.3mod} hold with 
\begin{equation}%\label{eq:}
\theta+\kF<1,\quad \theta+\kG<\frac{1}{2}-\frac{1}{Nq}.
\end{equation}
Then for all $\xi\in L^q(\Omega,\mathscr{F}_0,\mathbb{P}; E^c)$ there exists a unique global mild solution $X\in V_{T,q}$ of \eqref{eq:SCP}. Moreover, for some constant $C_{q,T}>0$ we have
\begin{equation}\label{eq:xvtfin}
\left\|X\right\|^q_{V_{T,q}}\leq C_{q,T}\left(1+\mathbb{E}\|\xi\|^q\right).
\end{equation}
\end{theo}
\begin{proof}
We only sketch a proof as it is analogous to the proofs of \cite[Thm.~4.3]{KvN12} and  \cite[Thm.~5.3]{Ce03} with highlighting the necessary changes. We set
\begin{equation}\label{eq:F_n}
F_n(t,\omega,u)\coloneqq\begin{cases}
F(t,\omega,u),& \text{ if }\|u\|\leq n,\\
F(t,\omega,\frac{nu}{\|u\|}),&\text{ otherwise.}
\end{cases}
\end{equation}
We argue in the same way as in the proof of \cite[Prop.~3.8(1)]{KvN12} which uses implicitly that, according to Assumption \ref{assum:mainvN4.3mod}(1), the Banach space $E$ is of type $2$, see also \cite[p.~978]{vNVW08}. The solution space will be $V_{T,q}$ defined in \eqref{eq:V_T} instead of $L^q\left(\Omega;C([0,T];E^c)\right)$ and $F_n+\widetilde{F}$ instead of $F_n$, we obtain that for each $n$ there exists a mild solution $X_n\in V_{T,q}$ of the problem \eqref{eq:SCP} with $F_n$ instead of $F$ (see also the proof of \cite[Thm.~5.3]{Ce03}). 
The mild solution $X_n$ satisfies, for all $t\in [0,T]$,  the integral equation
\begin{equation}\label{eq:mildsolXn}
X_n(t)=S(t)\xi+S\ast F_n(\cdot,X_n(\cdot))(t)+S\ast \widetilde{F}(\cdot,X_n(\cdot))(t)+S\diamond G(\cdot,X_n(\cdot))(t),
\end{equation}
almost surely.

We proceed in a completely analogous fashion as in the proof of \cite[Thm.~4.3]{KvN12}, with $E^c$ instead of $B$. Instead of \cite[Lem.~4.4]{KvN12} we use Lemma \ref{lem:44} with
\[u_n=X_n-S\diamond G(\cdot,X_n(\cdot)),\quad v_n=S\diamond G(\cdot,X_n(\cdot))\]
and obtain that
\begin{align}
\|u_n\|_{V_{T,q}}^q&=\|S(\cdot)\xi+S\ast F(\cdot,X_n(\cdot))+S\ast \widetilde{F}(\cdot,X_n(\cdot))\|_{V_{T,q}}^q\notag\\
&\lesssim\mathbb{E}\sup_{t\in [0,T]}\|S\ast \widetilde{F}(\cdot,X_n(\cdot))\|^q+\mathbb{E}\sup_{t\in [0,T]}\|S(t)\xi+\int_0^tS(t-s)F(s,u_n(s)+v_n(s))\ds\|^q\notag\\
&\lesssim \mathbb{E}\|S\ast \widetilde{F}(\cdot,X_n(\cdot))\|^{q}_{L^{\infty}(0,T;E^c)}\notag\\
&+C_{q,T}\left(1+\mathbb{E}\|\xi\|^q+\mathbb{E}\|S\diamond G(\cdot,X_n(\cdot))\|^{Nq}_{L^{\infty}(0,T;E^c)}\right).
\end{align}
Using that by Assumptions \ref{assum:mainvN4.3mod}(2) $E^{\theta}\hookrightarrow E^c$ holds, it follows from \cite[Lem.~3.6]{vNVW08} with $\alpha=1$, $\lambda=0$, $\eta=\theta$, $\theta=\kF$ that $S\ast \widetilde{F}(\cdot,X_n(\cdot))\in C([0,T],E^c)$ is satisfied and
\begin{equation}\label{eq:ScsillagFhullam}
\|S\ast \widetilde{F}(\cdot,X_n(\cdot))\|_{C([0,T],E^c)}\leq C(T)\cdot \|\widetilde{F}(\cdot,X_n(\cdot))\|_{L^{\infty}(0,T;E^{-\kF})}
\end{equation}
with $C(T)\to 0$ if $T\downarrow 0.$ Using \eqref{eq:Ftildegrow} and proceeding as in the proof of \cite[Thm.~4.3]{KvN12}, we obtain that for each $T>0$ there exists a constant $C_{q,T}>0$ such that 
\begin{equation}\label{eq:estimateXn}
\left\|X_n\right\|^q_{V_{T,q}}\leq C_{q,T}\left(1+\mathbb{E}\|\xi\|^q\right).
\end{equation}
We remark that the estimates needed use only the continuity of the embedding $E^{\theta}\hookrightarrow E^c.$
Once \eqref{eq:estimateXn} has been established we can conclude, the same way as in the proof of \cite[Thm.~5.3]{Ce03} the existence and uniqeness of a process 
$X\in L^q\left(\Omega;L^\infty(0,T;E^c\right))$ with
\begin{equation}\label{eq:xlqn}
\left\|X\right\|_{L^q\left(\Omega;L^\infty(0,T;E^c\right))}\leq C_{q,T}\left(1+\mathbb{E}\|\xi\|^q\right)
\end{equation}
such that for $t\in [0;T]$,
\begin{equation}\label{eq:mildsolX}
X(t)=S(t)\xi+S\ast F(\cdot,X(\cdot))(t)+S\ast \widetilde{F}(\cdot,X(\cdot))(t)+S\diamond G(\cdot,X(\cdot))(t),
\end{equation}
almost surely and thus $X$ is the unique mild solution of \eqref{eq:SCP} in $L^q\left(\Omega;L^\infty(0,T;E^c\right))$.

To prove the continuity of the trajectories of $X$ we first note that the analyticity of $S$ on $E^c$ immediately implies that $(0,T]\ni t\to S(t)\xi\in E^c$ is continuous. 
Next, we use \cite[Pro.~4.2.1]{Lun95} and the continuity of $F$ to obtain that
\[[0,T]\ni t\mapsto S\ast F(\cdot,X(\cdot))(t)\in E^c\]
is continuous. Using \cite[Lem.~3.6]{vNVW08} as above, we obtain that $S\ast \widetilde{F}(\cdot,X(\cdot))\in C([0,T],E^c)$ holds.
Applying analogous estimates as in the proof of \cite[Thm.~4.3]{KvN12}, we may conclude that that there exists $C(T)>0$ such that
\begin{equation}%\label{eq:}
\mathbb{E}\left\|S\diamond G(\cdot,X(\cdot))\right\|_{C([0,T];E^c)}^{Nq}\leq C(T)\left(1+\left\|X\right\|_{L^q\left(\Omega;L^\infty(0,T;E^c\right))}^q\right),
\end{equation}
hence it follows that
\[[0,T]\ni t\mapsto S\diamond G(\cdot,X(\cdot))(t)\in E^c\]
is continuous almost surely. Thus, by \eqref{eq:mildsolX} and the already established fact that $X\in L^q\left(\Omega;L^\infty(0,T;E^c\right))$, we obtain that $X\in V_{T,q}$ and by \eqref{eq:xlqn} the estimate \eqref{eq:xvtfin} holds.
\end{proof}

For our next result, introduce the following set of assumptions on the operators in (SCP). 
\begin{assum}$ $\label{assum:mainvN4.9mod}
\begin{enumerate}
	\item Let $E$ be a UMD Banach space of type $2$ and $(A, D(A))$ a densely defined, closed and sectorial operator on $E$.
	\item We have continuous (but not necessarily dense) embeddings for some $\theta\in (0,1)$ \[E^{\theta}\hookrightarrow E^c\hookrightarrow E.\]
	\item The strongly continuous analytic semigroup $S$ generated by $(A, D(A))$ on $E$ restricts to an analytic, contractive semigroup, denoted by $S^{c}$ on $E^c$, with generator $(A^c,D(A^c))$.
	\item The map $F\colon [0,T]\times\Omega\times E^c\to E^c$ is continuous in the first variable and locally Lipschitz continuous in the third variable in the sense that for all $r>0$, there exists a constant $L_{F}^{(r)}$ such that
	      \[\left\|F(t,\omega,u)-F(t,\omega,v)\right\|\leq L_{F}^{(r)}\|u-v\|\]
	      for all $\|u\|,\|v\|\leq r$ and $(t,\omega)\in [0,T]\times \Omega$ and there exists a constant $C_{F,0}\geq 0$ such that
				\[\left\|F(t,\omega,0)\right\|\leq C_{F,0},\quad t\in[0,T],\; \omega\in\Omega.\]
				Moreover, for all $u\in E^c$ the map $(t,\omega)\mapsto F(t,\omega,u)$ is strongly measurable and adapted.\\
				Finally, for suitable constants $a',b'\geq 0$ and $N\geq 1$ we have
				\[\langle A u+F (t,u+v), u^*\rangle\leq a'(1+\|v\|)^N+b'\|u\|\]
				for all $u\in D(A|_{E^c})$, $v\in E^c$ and $u^*\in\partial\|u\|,$ see \eqref{eq:subdiff}.
	\item There exist constants $a'',\, b'',\, k,\, K>0$ with $K\geq k$ such that the function $F\colon [0,T]\times\Omega\times E^c\to E^c$ satisfies
	      \begin{equation}\label{eq:assymF}
				\langle F (t,\omega, u+v)-F (t,\omega, v), u^*\rangle\leq a''(1+\|v\|)^K-b''\|u\|^k
				\end{equation}
				for all $t\in[0,T]$, $\omega\in\Omega$, $u,v\in E^c$ and $u^*\in\partial\|u\|,$ and
				\[\left\|F(t,v)\right\|\leq a''(1+\|v\|)^K\]
				for all $v\in E^c.$
	\item For some constant $\kF\geq 0$, the map $\widetilde{F}\colon [0,T]\times\Omega\times E^c\to E^{-\kF}$ is globally Lipschitz continuous in the third variable, uniformly with respect to the first and second variables.
				Moreover, for all $u\in E^c$ the map $(t,\omega)\mapsto \widetilde{F}(t,\omega,u)$ is strongly measurable and adapted.\\
				Finally, for some $d'\geq 0$ we have
				\begin{equation}\label{eq:Ftildegrow2}
				\left\|\widetilde{F}(t,\omega,u)\right\|_{E^{-\kF}}\leq d'\left(1+\|u\|\right)
				\end{equation}
				for all $(t,\omega,u)\in [0,T]\times \Omega\times E^c.$
	\item Let $\gamma(H,E^{-\kG})$ denote the space of $\gamma$-radonifying operators from $H$ to $E^{-\kG}$ for some $0\leq \kG<\frac{1}{2}$, see e.g. \cite[Sec.~3.1]{KvN12}.
	      Then the map $G\colon [0,T]\times\Omega\times E^c\to \gamma(H,E^{-\kG})$ is locally Lipschitz continuous in the sense that for all $r>0$, there exists a constant $L_{G}^{(r)}$ such that
	      \[\left\|G(t,\omega,u)-G(t,\omega,v)\right\|_{\gamma(H,E^{-\kG})}\leq L_{G}^{(r)}\|u-v\|\]
	      for all $\|u\|,\|v\|\leq r$ and $(t,\omega)\in [0,T]\times \Omega$. 
				Moreover, for all $u\in E^c$ and $h\in H$ the map $(t,\omega)\mapsto G(t,\omega,u)h$ is strongly measurable and adapted.\\
				Finally, for suitable constant $c',$
				\[\left\|G(t,\omega,u)\right\|_{\gamma(H,E^{-\kG})}\leq c'\left(1+\|u\|\right)^{\frac{k}{K}}\] 
				for all $(t,\omega,u)\in [0,T]\times \Omega\times E^c.$
\end{enumerate}
\end{assum}

\begin{rem}
Assumptions \ref{assum:mainvN4.9mod}(1)--(5) and (7) are -- in the first 3 cases slightly modified versions of -- Assumptions (A1), (A5), (A4), (F') and (G') in \cite{KvN12}. Assumption (6) is the assumption of \cite[Prop.~3.8]{KvN12} on $\widetilde{F}$. The main difference, besides the lack of strong continuity of $S$ on $E^c$ and that the embedding $E^{\theta}\hookrightarrow E^c$ is not necessarily dense, is that instead of (F'') we impose a possibly asymmetric growth condition \eqref{eq:assymF} on $F$. This is necessary so that later when we apply the abstract theory to \eqref{eq:stochnet} we may consider polynomial reaction terms with different degrees on different edges of the graph. The growth condition on $G$ in Assumption \ref{assum:mainvN4.9mod}(7) is also different from the linear growth condition on $G$ in (G'') of \cite{KvN12} as it reflects the possibly asymmetric  growth condition on $F$. It becomes a linear growth condition when $k=K$.
\end{rem}

The following result is analogous to the statement of \cite[Lem.~4.8]{KvN12} but with the semigroup $S^c$ being an analytic contraction semigroup on $E^c$ which is not necessarily strongly continuous and with the asymmetric growth condition \eqref{eq:assymF} on $F$. Again, the main difference in the proof is the use of a different approximation argument as the Yosida approximation argument in \cite{KvN12} uses the strong continuity of $S^c$ on $E^c$ (the latter denoted by $B$ there) in a crucial manner.

\begin{lemma}\label{lem:48}
Let $S^c$ be an analytic contraction semigroup on $E^c$. Let $x\in E^c$ and $F\colon [0,T]\times\Omega\times E^c\to E^c$ satisfy condition (4) and (5) of Assumptions \ref{assum:mainvN4.9mod}. Assume that $u\in C((0,T];E^c)\cap L^{\infty}(0,T;E^c)$ and $v\in C([0,T]\;E^c)$  satisfy
\begin{equation}\label{eq:lem48mildsol}
u(t)=S^c(t)x+\int_0^tS^c(t-s)F(s,u(s)+v(s))\ds,\quad \forall t\in[0,T].
\end{equation}
Then
\begin{equation}\label{eq:lem48}
\sup_{t\in[0,T]}\left\|u(t)\right\|\leq \|x\|+c\left(1+\sup_{t\in[0,T]}\|v(t)\|\right)^{\frac{K}{k}}
\end{equation}
with $c=\left(\frac{4a''}{b''}\right)^{\frac{1}{k}}$
\end{lemma}
\begin{proof}
We will proceed similarily as in Lemma \ref{lem:44}.
We denote by $A^c$ the generator of $S^c$, being sectorial and dissipative (see Remark \ref{rem:analdisscontr}) and fix $v\in C([0,T]\;E^c)$ satisfying \eqref{eq:lem48mildsol}. Hence we can take $\la\in\rho(A^c)$, and introduce the problem on $E^c$
\begin{equation}\label{eq:pflem48la}
\dot{z}(t)=A^c z(t)+F(t,z(t)+v(t)),\quad z(0)=\la R(\la,A^c)x\eqqcolon x_{\la}.
\end{equation}
As we have showed in the proof of Lemma \ref{lem:44}, there exists a unique global solution $u_{\la}\in C([0,T],E^c))$ of \eqref{eq:pflem48la}, satisfying
\begin{equation}\label{eq:pflem48ula}
u_{\la}(t)=S^c(t)x_{\la}+\int_0^tS^c(t-s)F(s,u_{\la}(s)+v(s))\ds\eqqcolon S^c(t)x_{\la}+\int_0^tS^c(t-s)f_{\la}(s)\ds,
\end{equation}
where $f_{\la}(s)=F(s,u_{\la}(s)+v(s)).$ Since $f_{\la}(\cdot)\in C([0,T],E^c)$ and by definition, $u_{\la}(0)=x_{\la}\in D(A^c)$ holds, we can apply \cite[Prop.~4.1.8]{Lun95} and obtain that $u_{\la}$ is a strong solution of \eqref{eq:pflem48la} in the sense of \cite[Def.~4.1.1]{Lun95}. This means, that there exists a sequence $(u_{\la,n})\subset C^1([0,T],E^c)\cap C([0,T],D(A^c))$ such that
\begin{equation}\label{eq:lem48approx}
u_{\la,n}\to u_{\la},\quad \dot{u}_{\la,n}-A^c u_{\la,n}\eqqcolon f_{\la,n}\to f_{\la}\text{ in }C([0,T],E^c)
\end{equation}
holds, as $n$ goes to infinity. 
Using that $u_{\la,n}\in C^1([0,T],E^c)$, we have by \cite[Thm.~17.9]{HeStr75} that
\[\frac{d}{dt}\|u_{\la,n}(t)\|=\langle \dot{u}_{\la,n}(t),u_{\la,n}(t)^*\rangle,\text{ for all }u_{\la,n}(t)^*\in\partial \|u_{\la,n}(t)\|.\]
Hence, for all $u_{\la,n}(t)^*\in\partial \|u_{\la,n}(t)\|$,
\begin{align}
\frac{d}{dt}\|u_{\la,n}(t)\|&=\langle A^c u_{\la,n}(t)+f_{\la,n}(t),u_{\la,n}(t)^*\rangle\\
&=\langle A^c u_{\la,n}(t),u_{\la,n}(t)^*\rangle\\
&+\langle F(t,u_{\la,n}(t)+v(t))-F(t,v(t)),u_{\la,n}(t)^*\rangle+\langle F(t,v(t)),u_{\la,n}(t)^*\rangle\\
&+\langle f_{\la}(t)-F(t,u_{\la,n}(t)+v(t)),u_{\la,n}(t)^*\rangle+\langle f_{\la,n}(t)-f_{\la}(t),u_{\la,n}(t)^*\rangle
\end{align}
Using now the dissipativity of $A^c$ (see Remark \ref{rem:analdisscontr}) and the assumptions on $F$, we obtain that
\begin{align}\label{eq:lem48ulanest}
\frac{d}{dt}\|u_{\la,n}(t)\|&\leq 2a''\left(1+\|v(t)\|\right)^K-b''\|u_{\la,n}(t)\|^k\notag\\
&+\|f_{\la}-F(\cdot,u_{\la,n}+v)\|_{C([0,t],E^c)}+\|f_{\la,n}-f_{\la}\|_{C([0,t],E^c)}\notag\\
&\leq 2a''\left(1+\sup_{s\in[0,T]}\|v(t) \|\right)^K-b''\|u_{\la,n}(t)\|^k+\ve_{\la,n}
\end{align}
with 
\begin{equation}\label{eq:lem48epsilon}
\ve_{\la,n}=\|f_{\la}-F(\cdot,u_{\la,n}+v)\|_{C([0,t],E^c)}+\|f_{\la,n}-f_{\la}\|_{C([0,t],E^c)}.
\end{equation}
By the continuity of $F$ and \eqref{eq:lem48approx}, we have that 
\begin{equation}\label{eq:eq:lem48epsilonto0}
\ve_{\la,n}\to 0,\quad n\to\infty.
\end{equation}\smallskip

We now fix $n\in\nat$ and define $\varphi(t)\coloneqq \|u_{\la,n}(t)\|$ and 
\[\gamma\coloneqq \left(2a''\left(1+\sup_{s\in[0,T]}\|v(s)\|\right)^{K}+\ve_{\la,n}\right)^{\frac{1}{k}}.\] 
Then $\varphi$ is absolutely continuous and by \eqref{eq:lem48ulanest}
\[\varphi'(t)\leq -b''\varphi(t)^k+\gamma^k\]
holds almost everywhere. We will prove that
\begin{equation}\label{eq:lem48ulan}
\varphi(t)\leq \|x_{\varphi}\|+\left(\frac{2}{b''}\right)^{\frac{1}{k}}\cdot \gamma,\quad t\in[0,T],
\end{equation}
where $x_{\varphi}=u_{\la,n}(0)$. Assume to the contrary that for some $t_0\in [0,T]$
\begin{equation}\label{eq:pflem48contr}
\varphi(t_0)> \|x_{\varphi}\|+ \left(\frac{2}{b''}\right)^{\frac{1}{k}}\cdot \gamma.
\end{equation}
Since $\varphi(0)=\|x_{\varphi}\|$, we have that $t_0\in (0,\delta]$. 
Let $\psi:I\to\real$ be the unique maximal solution of
\[\begin{cases}
\psi'(t)&=-b''\psi(t)^k+\gamma^k,\\
\psi(t_0)&=\varphi(t_0)-\|x_{\varphi}\|.
\end{cases}
\]
We can use \cite[Cor.~4.7]{KvN12} with $u^+(t)=\psi(t)$, $u^-(t)=\varphi(t)$, 
\[f(t,u)=-b''u^k+\gamma^k\] 
and the assumption $u^+(t_0)+\|x_{\varphi}\|\leq u^-(t_0)$. This implies that $u^+(t)+\|x\|\leq u^-(t)$, that is, 
\begin{equation}\label{eq:pflem48}
\psi(t)+\|x_{\varphi}\|\leq \varphi(t),\text{ for all }t\in I\cap [0,t_0].
\end{equation}
Using the same arguments as in the proof of \cite[Lem.~4.8]{KvN12} we obtain that $0\in I$ and
\[\psi(t)>\left(\frac{1}{b''}\right)^{\frac{1}{k}}\cdot \gamma,\quad t\in I\cap [0,t_0].\]
This implies by the definition of $\psi$ that $\psi'(t)<0$, hence $\psi$ is decreasing. Combining this together with \eqref{eq:pflem48} and \eqref{eq:pflem48contr} we obtain
\[\|x_{\varphi}\|=\varphi(0)\geq \psi(0)+\|x_{\varphi}\|\geq \psi(t_0)+\|x_{\varphi}\|=\varphi(t_0)>\left(\frac{2}{b''}\right)^{\frac{1}{k}}\cdot \gamma+\|x_{\varphi}\|,\]
hence
\[0>\left(\frac{2}{b''}\right)^{\frac{1}{k}}\cdot \gamma,\]
which is a contradiction.\smallskip

Hence, we have proved \eqref{eq:lem48ulan}. Since $n$ was arbitrary, we obtain that for all $n\in\nat$
\begin{equation}%\label{eq:}
\|u_{\la,n}(t)\|\leq \|u_{\la,n}(0)\|+\left(\frac{2}{b''}\right)^{\frac{1}{k}}\cdot \left(2a''\left(1+\sup_{s\in[0,T]}\|v(s)\|\right)^{K}+\ve_{\la,n}\right)^{\frac{1}{k}},\quad t\in[0,T].
\end{equation}
Letting $n\to\infty$, by \eqref{eq:lem48approx} and \eqref{eq:eq:lem48epsilonto0} we obtain
\begin{equation}\label{eq:pflem48ulaest}
\|u_{\la}(t)\|\leq \|x_{\lambda}\|+\left(\frac{4a''}{b''}\right)^{\frac{1}{k}}\cdot \left(1+\sup_{s\in[0,T]}\|v(s)\|\right)^{\frac{K}{k}},\quad t\in[0,T].
\end{equation}
Finally, using the same argument as at the end of the proof of \cite[Prop.~6.2.2]{Ce01b}, we obtain that for any $t$, the net $\{u_{\la}(t)\}_{\la\in\rho(A^c)}$ is a Cauchy-net in $E^c$, hence it is convergent and the limit is $u(t).$ This yields \eqref{eq:lem48}.
\end{proof}

The next result is a generalized version of that of Kunze and van Neerven which was first proved in \cite[Thm.~4.9]{KvN12} but with a typo in the statement and was later corrected in the recent arXiv preprint \cite[Thm.~4.9]{KvN19}.

\begin{theo}\label{theo:KvN4.9gen}
Let $T>0$, $2<q<\infty$ and suppose that Assumptions \ref{assum:mainvN4.9mod} hold with 
\[\theta+\kF<1,\quad \theta+\kG<\frac{1}{2}-\frac{1}{q}.\]
Then for all $\xi\in L^q(\Omega,\mathscr{F}_0,\mathbb{P}; E^c)$ there exists a global mild solution $X\in V_{T,q}$ of \eqref{eq:SCP}.  Moreover, for some constant $C_{q,T}>0$ we have
\[\|X\|_{V_{T,q}}^q\leq C_{q,T}\left(1+\mathbb{E}\|\xi\|^q\right).\]
\end{theo}
\begin{proof}
We can proceed similarly as in the proofs of \cite[Thm.~4.9]{KvN12} and \cite[Thm.~5.9]{Ce03}. We set
\begin{equation}\label{eq:G_n}
G_n(t,\omega,u)\coloneqq\begin{cases}
G(t,\omega,u),& \text{ if }\|u\|\leq n,\\
G(t,\omega,\frac{nu}{\|u\|}),&\text{ otherwise.}
\end{cases}
\end{equation}
We obtain by Theorem \ref{theo:KvN4.3gen} that for each $n$ there exists a global mild solution $X_n\in V_{T,q}$ of the problem \eqref{eq:SCP} with $G_n$ instead of $G$ (see also the proof of \cite[Thm.~5.5]{Ce03}). Using \eqref{eq:mildsolXn} for $X_n$ and setting
\[u_n=X_n-S\diamond G_n(\cdot,X_n(\cdot)),\quad v_n=S\diamond G_n(\cdot,X_n(\cdot))\]
we obtain that
\begin{align}\label{eq:pfKvN4.9gen1}
\|u_n\|_{V_{T,q}}^q&=\|S(\cdot)\xi+S\ast F(\cdot,X_n(\cdot))+S\ast \widetilde{F}(\cdot,X_n(\cdot))\|_{V_{T,q}}^q\notag\\
&\lesssim\mathbb{E}\sup_{t\in [0,T]}\|S\ast \widetilde{F}(\cdot,X_n(\cdot))\|^q+\mathbb{E}\sup_{t\in [0,T]}\|S(t)\xi+\int_0^tS(t-s)F(s,u_n(s)+v_n(s))\ds\|^q\notag\\
&\lesssim C(T)\cdot \mathbb{E}\|\widetilde{F}(\cdot,X_n(\cdot))\|^{q}_{L^{\infty}(0,T;E^{-\kF})}+\mathbb{E}\|\xi\|^q\\
&+\left(\frac{4a''}{b''}\right)^{\frac{1}{k}}\left(1+\mathbb{E}\|S\diamond G_n(\cdot,X_n(\cdot))\|_{L^{\infty}(0,T;E^c)}^{\frac{K}{k}q}\right),\notag
\end{align}
where $\lesssim$ denotes that the expression on the left-hand-side is less or equal to a constant times the expression on the right-hand-side. In the last inequality we have used estimate \eqref{eq:ScsillagFhullam} with $C(T)\to 0$ if $T\downarrow 0$ and Lemma \ref{lem:48} with $u=u_n$ and $v=v_n$. 

As in the proof of \cite[Thm.~4.3]{KvN12} with $E^c$ instead of $B$, we obtain that for each $T>0$ there exist a constant $C'_{T}>0$ such that
\begin{align}\label{eq:estimateGn}
\mathbb{E}\|S\diamond G_n(\cdot,X_n(\cdot))\|_{C([0,T];E^c)}^{\frac{K}{k}q}&\lesssim \mathbb{E}\int_0^T\|G_n(t,X_n(t))\|_{\gamma(H,E^{-\kG})}^{\frac{K}{k}q}\dt\\
&\lesssim\mathbb{E}\int_0^T \left(1+\|X_n(t)\|^q\right)\dt\notag\\
&\leq C'(T)\left(1+\|X_n\|_{V_{T,q}}^q\right),\notag
\end{align}
where in the second inequality we have used Assumptions \ref{assum:mainvN4.9mod}(7) and we have $C'(T)\to 0$ if $T\downarrow 0$. 
Combining this with \eqref{eq:Ftildegrow2} and \eqref{eq:pfKvN4.9gen1} we obtain that there are positive constants $C_0$, $C_1$ and $C_2(T)$ such that
\[\|X_n\|_{V_{T,q}}^q\lesssim \|u_n\|_{V_{T,q}}^q+\|v_n\|_{V_{T,q}}^q\leq C_0(T)+C_1 \mathbb{E}\|\xi\|^q+ C_2(T)\|X_n\|_{V_{T,q}}^q\]
with $C_2(T)\to 0$ as $T\downarrow 0.$ The proof can be finished as that of Theorem \ref{theo:KvN4.3gen}.
\end{proof}

%%%%%%%%%%%%%%%%%%%%%%%%%%%%%%%%%%%%%%%%%%%%%%%%%%%%%
\section{A stochastic reaction-diffusion equation}\label{sec:srden}
%%%%%%%%%%%%%%%%%%%%%%%%%%%%%%%%%%%%%%%%%%%%%%%%%%%%%

%%%%%%%%%%%%%%%%%%%%%%%%%%%%%%%%%%%%%%%%
\subsection{Preparatory results}\label{sec:srdenprep}

%%%%%%%%%%%%%%%%%%%%%%%%%%%%%%%%%%%%%%%

In order to apply the abstract result of Theorem \ref{theo:KvN4.9gen} to the stochastic reaction-diffusion equation we need to prove some preliminary results regarding the setting of Section \ref{sec:determnetwork}. We make use of the fact that the semigroups involved here all leave the corresponding real spaces invariant (this follows from the first bullet in the proof of Lemma \ref{lem:subMarkov} and the corresponding Beurling--Deny criterion).

\begin{defi}\label{defi:mcEc}
We denote by
\begin{equation}
\mcE^c:=(C[0,1])^m\times\real^n
\end{equation}
the product space of continuous functions on the edges (not necessarily continuous in the vertices) and denote its elements by
\[U=\left(\begin{smallmatrix}u\\ r\end{smallmatrix}\right)\in\mcE^c\text{ with }u\in (C[0,1])^m,\; r\in\real^n. \]
The norm is defined as usual with
\[\left\|U\right\|_{\mcE^c}\coloneqq \max\left\{\|u\|_{(C[0,1])^m},\|r\|_{\ell^{\infty}}\right\},\quad U=\left(\begin{smallmatrix}u\\ r\end{smallmatrix}\right)\in\mcE^c.\]
\end{defi}
This space will play the role of the space $E^c$ in our setting. We recall that for $p\in[1,\infty]$ the operators $(\mcA_p,D(\mcA_p))$ are generators of analytic semigroups (see Proposition \ref{prop:sgrextend}) on the spaces $\mcE_p$ defined in \eqref{eq:mcEp}.
\begin{equation}\label{eq:fractdommcE}
\text{For }0 \leq \theta< 1\text{ let }\mcE_p^{\theta}\text{ be defined as in \eqref{eq:fractdom} for the operator }\mcA_p\text{ on the space }\mcE_p. 
\end{equation}
We will need the following result on the fractional power spaces $\mcE_p^{\theta}.$
\begin{lemma}\label{lemma:fractionalspaceiso}
For the fractional domain spaces $\mcE_p^{\theta}$ defined in \eqref{eq:fractdommcE} and $1< p<\infty$ arbitrary we have that
\begin{enumerate}
	\item if $0<\theta< \frac{1}{2p}$, then
	\[\mcE_p^{\theta}\cong \left(\prod _{j=1}^m W^{2\theta,p}(0,1;\mu_j dx)\right)\times \real^n;\]
	\item if $\theta > \frac{1}{2p}$, then
	\[\mcE_p^{\theta} \cong \left(\prod _{j=1}^m W_{0}^{2\theta,p}(0,1;\mu_j dx)\right)\times \real^n.\] 
\end{enumerate}
\end{lemma}
\begin{proof}
By Proposition \ref{prop:sgrextend} the operator $(\mcA_p,D(\mcA_p))$ generates a positive, contraction semigroup on $\mcE_p$.
Hence, we can use  \cite[Thm.~in $\mathsection$4.7.3]{Ar04} (see also \cite{Duo89}) obtaining that for any $\omega'>0$, $\omega'-\mcA_p$ has a bounded $H^{\infty}(\Sigma_{\varphi})$-calculus for each $\varphi>\frac{\pi}{2}$. Proposition \ref{prop:sgrextend} implies that $\omega'-\mcA_p$ is injective and sectorial thus it has bounded imaginary powers (BIP). Therefore, by \cite[Prop.~in $\mathsection$4.4.10]{Ar04} (see also \cite[Thm.~11.6.1]{MS01}), it follows that for the complex interpolation spaces
\begin{equation}\label{eq:inclpf2}
\mcE_p^{\theta}=D((\omega'-\mcA_p)^{\theta})\cong [D(\omega'-\mcA_p),\mcE_p]_{\theta}\cong [D(\mcA_p),\mcE_p]_{\theta}
\end{equation}
holds with equivalence of norms. Denote
\[W_0(G)\coloneqq \prod _{j=1}^m W_{0}^{2,p}(0,1;\mu_j dx),\]
where $W_{0}^{2,p}(0,1;\mu_j dx)=W^{2,p}(0,1;\mu_j dx)\cap W_{0}^{1,p}(0,1;\mu_j dx),$ $j=1,\dots m.$ Hence, $W_0(G)$ contains such vectors of functions that are twice weakly differentiable on each edge and continuous in the vertices with Dirichlet boundary conditions. By 
\cite[Cor.~3.6]{Mu14a},
\begin{equation}\label{eq:inclpf3}
D(\mcA_p)\cong W_0(G)\times \real^n,
\end{equation}
where the isomorphism is established by a similarity transform of $\mcE_p$. Using general interpolation theory, see e.g.~\cite[Sec.~4.3.3]{Triebel78}, we have that if $\theta < \frac{1}{2p}$, then
\begin{equation}\label{eq:inclpf4}
\left[W_0(G)\times \real^n,\mcE_p\right]_{\theta}\cong \left(\prod _{j=1}^m W^{2\theta,p}(0,1;\mu_j dx)\right)\times \real^n.
\end{equation}
Hence, summing up \eqref{eq:inclpf2}, \eqref{eq:inclpf3} and \eqref{eq:inclpf4}, for $\theta < \frac{1}{2p}$
\[\mcE_p^{\theta}\cong \left(\prod _{j=1}^m W^{2\theta,p}(0,1;\mu_j dx)\right)\times \real^n\]
holds.

Furthermore, using again \cite[Sec.~4.3.3]{Triebel78}, for $\theta>\frac{1}{2p}$ we have that
\begin{equation}\label{eq:inclpf5}
\left[W_0(G)\times \comp^n,\mcE_p\right]_{\theta}\cong \left(\prod _{j=1}^m W_{0}^{2\theta,p}(0,1;\mu_j dx)\right)\times \real^n.
\end{equation}
Hence, by \eqref{eq:inclpf2}, \eqref{eq:inclpf3} and \eqref{eq:inclpf5},  for $\theta>\frac{1}{2p}$
\begin{equation}
\mcE_p^{\theta} \cong \left(\prod _{j=1}^m W_{0}^{2\theta,p}(0,1;\mu_j dx)\right)\times \real^n
\end{equation}
holds.
\end{proof}

\begin{cor}\label{cor:fractionalspaceincl}
For $\theta> \frac{1}{2p}$ the following continuous embeddings are satisfied:
\begin{equation}\label{eq:Ethetaincl}
\mcE_p^{\theta}\hookrightarrow\mcE^c\hookrightarrow \mcE_p.
\end{equation} 
\end{cor}
\begin{proof}
According to Lemma \ref{lemma:fractionalspaceiso}$(2)$ we have that for $\theta> \frac{1}{2p}$
\begin{equation}\label{eq:interpol}
\mcE_p^{\theta} \cong \left(\prod _{j=1}^m W_{0}^{2\theta,p}(0,1;\mu_j dx)\right)\times \real^n
\end{equation}
holds. Hence, by Sobolev imbedding we obtain that for $\theta>\frac{1}{2p}$
\[\mcE_p^{\theta}\hookrightarrow \mcE^c\]
is satisfied. The claim follows by observing $\mcE^c\hookrightarrow\mcE_p$.
\end{proof}

In the following we will prove that each of the semigroups $(\mcT_p(t))_{t\geq 0}$ restricts to the same analytic semigroup of contractions on $\mcE^c$.

\begin{prop}\label{prop:mcAonmcEc}
For all $p\in[1,\infty]$, the semigroups $\mcT_p$ leave $\mcE^c$ invariant, and the restrictions $\mcT_p|_{\mcE^c}$ all coincide that we denote by $\mcS^c$. The semigroup $\mcS^c$ is analytic and contractive on $\mcE^c$. Its generator $(\mcA^c,D(\mcA^c))$ coincides with the part $(\mcA_p|_{\mcE^c},D(\mcA_p|_{\mcE^c}))$ of the operator $(\mcA_p,D(\mcA_p))$ in $\mcE^c$ for any $p\in[1,\infty]$.
\end{prop}
\begin{proof}
First we will show that for each $p\in[1,\infty]$, $D(\mcA_p)\subset \mcE^c$ holds. If $p\in [1,\infty)$ it follows easily from \eqref{eq:mcAp} and Sobolev imbedding. For $p=\infty$ take $U\in D(\mcA_{\infty})$. Then for any $\la>0$ there exists $V\in \mcE_{\infty}$ such that $R(\la,\mcA_{\infty})V=U$. Using that the semigroup $\mcT_2$ is the extension of $\mcT_{\infty}$ to $\mcE_2$ by \eqref{eq:sgrconsistent} and $\mcE_{\infty}\hookrightarrow\mcE_2$ holds, by a similar argument as in Remark \ref{rem:analdisscontr} we obtain that $V\in \mcE_2$ and $R(\la,\mcA_{2})V=U\in D(\mcA_2)$. The claim follows now by observing $D(\mcA_2)\subset \mcE^c$.

From Proposition \ref{prop:sgrextend} we know that for each $p\in[1,\infty]$ the semigroup $\mcT_p$ is analytic and contractive. Hence, using the inclusion $D(\mcA_p)\subset \mcE^c$ and \cite[Thm.~3.7.19]{ABHN11}, we obtain that $\mcT_p$ leaves $\mcE^c$ invariant. By \eqref{eq:sgrconsistent} we also have that the restrictions on $\mcE^c$ all coincide, thus we may use $\mcS^c$ to denote this common restriction. It is straighforward that $\mcS^c$ is a contraction semigroup on $\mcE^c$ since $\mcT_{\infty}$ is a contraction semigroup on $\mcE_{\infty}$ and the norms on $\mcE^c$ and $\mcE_{\infty}$ coincide.

Using the same argument as in Remark \ref{rem:analdisscontr}, and the fact $D(\mcA_p)\subset \mcE^c$, we obtain that $\mcA^c=\mcA_p|_{\mcE^c}$ for all $p\in[1,\infty]$.

It remains to prove that $\mcS^c$ is analytic. We now use that $\mcT_{\infty}$ is analytic on $\mcE_{\infty}$. That is, by \cite[Cor.~3.7.18]{ABHN11}, there exists $r>0$ such that $\left\{is:s\in\real,\; |s|>r\right\}\subset\rho(\mcA_{\infty})$
and
\begin{equation}\label{eq:resestimate}
\sup_{|s|>r}\left\|sR(is,\mcA_{\infty})\right\|_{\mcE_{\infty}}<\infty.
\end{equation}
Since $\mcA_{\infty}|_{\mcE^c}=\mcA^c$, the analogue of \eqref{eq:resestimate} also holds with $\mcA^c$ instead of $\mcA_{\infty}$ and with respect to the norm of $\mcE^c$. Hence, using \cite[Cor.~3.7.18]{ABHN11} again, we obtain that the semigroup $\mcS^c$ generated by $\mcA^c$ on $\mcE^c$ is analytic. \end{proof}

%%%%%%%%%%%%%%%%%%%%%%%%%%%%%%%%%%%%%%%%%%%%%%%%%%%%%%
\subsection{Main results}\label{subsec:mainresults}
%%%%%%%%%%%%%%%%%%%%%%%%%%%%%%%%%%%%%%%%%%%%%%%%%%%%%%%

We now apply the results of the previous sections to the following stochastic evolution equation, based on \eqref{netcp} (see \cite[Sec.~2]{BMZ08}, \cite[Sec.~5]{KvN12}). We prescribe stochastic noise on the nodes as well as on the edges of the network.

Let $(\Omega,\mathscr{F},\mathbb{P})$ be a complete probability space endowed with a right-continuous filtration $\mathbb{F}=(\mathscr{F}_t)_{t\in [0,T]}$ for some $T>0$ given. We consider the problem
\begin{equation}\label{eq:stochsys2}
\left\{\begin{array}{rcll}
\dot{u}_j(t,x)&=& (c_j u_j')'(t,x)+d_j(x)\cdot u_j'(t,x)&\\
&-& p_j(x)u_j(t,x)+f_j(t,x,u_j(t,x))&\\
&+&h_j(t,x,u_j(t,x))\frac{\partial w_j}{\partial t}(t,x), &t\in(0,T],\; x\in(0,1),\; j=1,\dots,m, \\
u_j(t,\mv _i)&=&u_\ell (t,\mv _i)\eqqcolon r_i(t), &t\in(0,T],\; \forall j,\ell\in \Gamma(\mv _i),\; i=1,\ldots,n,\\
\dot{r}_i(t)&=&[M r(t)]_{i}&\\
&&+\sum_{j=1}^m \phi_{ij}\mu_{j} c_j(\mv_i) u'_j(t,\mv_i)&\\
&&+g_i(t,r_i(t))\dot{\beta}_i(t), &t\in(0,T],\; i=1,\ldots,n,\\
u_j(0,x)&=&\mathsf{u}_{j}(x), &x\in [0,1],\; j=1,\dots,m. \\
r_i(0)&=&\mathsf{r}_{i}, & i=1,\dots ,n.
\end{array}
\right.
\end{equation}

Here  $\dot{\beta}_i(t)$, $i=1,\dots ,n$, are independent noises; written as formal derivatives of independent scalar Brownian motions $(\beta_i(t))_{t\in [0,T]}$, defined on $(\Omega,\mathscr{F},\mathbb{P})$ with respect to the filtration $\mathbb{F}$. The terms $\frac{\partial w_j}{\partial t}$, $j=1,\dots ,m$, are independent space-time white noises on $[0,1]$; written as formal derivatives of independent cylindrical Wiener-processes $(w_j(t))_{t\in [0,T]}$, defined on $(\Omega,\mathscr{F},\mathbb{P})$, in the Hilbert space $L^2(0,1; \mu_j dx)$  with respect to the filtration $\mathbb{F}$.

In contrast to Section \ref{sec:determnetwork}, we add a first order term $d_j(x)\cdot u_j'(t,x)$ to the first equation of \eqref{netcp} assuming
\begin{equation}%\label{eq:}
d_j\in \mathrm{Lip}[0,1],\quad j=1,\dots ,m.
\end{equation}

The functions $f_j\colon [0,T]\times \Omega\times [0,1]\times \real\to\real$ are polynomials of the form
\begin{equation}\label{eq:fjdef}
f_j(t,\omega,x,\eta)=-a_j(t,\omega,x)\eta^{2k_j+1}+\sum_{l=0}^{2k_j}a_{j,l}(t,\omega,x)\eta^l,\quad \eta\in\real,\, j=1,\dots ,m
\end{equation} 
for some integers $k_j,$ $j=1,\dots ,m.$
For the coefficients we assume that there are constants $0<c\leq C<\infty$ such that
\begin{equation}\label{eq:assa_j}
c\leq a_j(t,\omega,x)\leq C,\;\left|a_{j,l}(t,\omega,x)\right|\leq C,\text{ for all }j=1,\dots ,m,\;l=0,\dots ,2k_j,
\end{equation}
for all $(t,\omega,x)\in [0,T]\times \Omega\times [0,1]$, see \cite[Ex.~4.2]{KvN12}. Furthermore, we suppose that
\[a_j(t,\omega,\cdot),\,a_{j,l}(t,\omega,\cdot)\in C[0,1],\quad j=1,\dots m,\;l=0,\dots ,2k_j\]
and that the coefficients $a_{j,l}\colon [0,T]\times \Omega\times [0,1]\to\real$ are jointly measurable and adapted in the sense that for each $j$ and $l$ and for each $t\in[0,T]$, the function $a_{j,l}(t,\cdot)$ is $\mathscr{F}_t\otimes \mcB_{[0,1]}$-measurable, where $\mcB_{[0,1]}$ denotes the sigma-algebra of the Borel sets on $[0,1].$

\begin{rem}
The functions coming from the classical FitzHugh-Nagumo problem (see e.g. \cite{BMZ08})
\begin{equation}%\label{eq:}
f_j(\eta)\coloneqq \eta(\eta-1)(a_j-\eta),\quad j=1,\dots ,m
\end{equation}
with $a_j\in (0,1)$ satisfy the conditions above.
\end{rem}

Furthermore, let
\begin{align}\label{eq:defkK}
K\coloneqq 2k_{max}+1,\quad k\coloneqq 2k_{min}+1,\\
\text{ where }k_{max}=\max_{j=1,\dots,m}k_j,\quad k_{min}=\min_{j=1,\dots,m}k_j.\notag
\end{align}

For the functions $g_i$ we assume
\begin{align}
&g_i\colon [0,T]\times \Omega \times\real\to \real,\quad i=1,\dots ,n \text{ are locally Lipschitz continuous}\notag\\
&\text{in the third variable, uniformly with respect to the first 2 variables, and}\label{eq:gidefLip}\\
&|g_i(t,\omega,r)|\leq c(1+|r|)^{\frac{k}{K}}\text{ for all }(t,\omega,r)\in [0,T]\times\Omega\times \real\label{eq:gidefnov}
\end{align}
where the constants $k$ and $K$ are defined in \eqref{eq:defkK}. We further require that the functions $g_i$ are jointly measurable and adapted in the sense that for each $i$ and $t\in[0,T]$, $g_i(t,\cdot)$ is $\mathscr{F}_t\otimes \mcB_{\real}$-measurable, where  $\mcB_{\real}$ denotes the sigma-algebra of the Borel sets on $\real$.\smallskip

We suppose that
\begin{align}%\label{eq:hjdef}
&h_j\colon [0,T]\times \Omega\times [0,1]\times \real\to\real,\quad j=1,\dots ,m \text{ are locally Lipschitz continuous}\notag\\
&\text{in the fourth variable, uniformly with respect to the first 3 variables, and}\label{eq:hjdefLip}\\
&|h_j(t,\omega, x,\eta)|\leq c(1+|\eta|)^{\frac{k}{K}}\text{ for all }(t,\omega, x,\eta)\in [0,T]\times [0,1]\times \real.\label{eq:hjdefnov}
\end{align}
We further assume that the functions $h_j$ are jointly measurable and adapted in the sense that for each $j$ and $t\in[0,T]$, $h_j(t,\cdot)$ is $\mathscr{F}_t\otimes \mcB_{[0,1]}\otimes \mcB_{\real}$-measurable, where $\mcB_{[0,1]}$ and $\mcB_{\real}$ denote the sigma-algebras of the Borel sets on $[0,1]$ and $\real$, respectively.

We rewrite system \eqref{eq:stochsys2} in an abstract form analogously to \eqref{eq:SCP}

\begin{equation}\tag{SCPn}\label{eq:SCPn}
\left\{
\begin{aligned}
d\mcX(t)&=[\mcA \mcX(t)+\mcF(t,\mcX(t))+\widetilde{\mcF}(t,\mcX(t))]dt+\mcG(t,\mcX(t))d\mcW(t)\\
\mcX(0)&=\xi.
\end{aligned}
\right.
\end{equation}

The operator $(\mcA, D(\mcA))$ is $(\mcA_p, D(\mcA_p))$ for some large $p\in [2,\infty)$, where $p$ will be chosen in \eqref{eq:thetap} \eqref{eq:mainthmqp} and \eqref{eq:Holderregpfp}, \eqref{eq:mainthmBp2} later. Hence, by Proposition \ref{prop:sgrextend}, $\mcA$ is the generator of the strongly continuous analytic semigroup $\mathcal{S}\coloneqq \mcT_p(t)$ on the Banach space $\mcE_p$, and $\mcE_p$ is a UMD space of type $2$. 

For the function $\mcF\colon [0,T]\times \Omega\times\mcE^c\to \mcE^c$ we have
\begin{equation}\label{eq:mcFdef}
\mcF(t,\omega,U)\coloneqq \begin{pmatrix}
	\mcF_1(t,\omega,u) \\ 0_{\real^n}
\end{pmatrix},\quad U=\begin{pmatrix}
	u \\ r
\end{pmatrix}\in \mcE^c,
\end{equation}
with 
\begin{align}\label{eq:mcF1}
\mcF_1(t,\omega,u)&\colon [0,1]\to \real^m,\notag\\
(\mcF_1(t,\omega,u))(x)&\coloneqq \left(f_1(t,\omega,x,u_1(x)),\dots ,f_m(t,\omega,x,u_m(x))\right)^{\top},
\end{align}
 $t\in[0,T]$, $\omega\in\Omega$.

We now define $\widetilde{\mcF}$ and $\mcG$ and after that we prove that they map between the appropriate spaces as assumed in Section \ref{sec:srde}. Let
\begin{equation}\label{eq:mcFtilde}
(\widetilde{\mcF} U)(x)\coloneqq \begin{pmatrix}
	\widetilde{\mcF}_1 u \\ 0_{\real^n}
\end{pmatrix},\quad U=\begin{pmatrix}
	u \\ r
\end{pmatrix}\in (C^1[0,1])^m\times\real^n,
\end{equation}
be defined as a map from $(C^1[0,1])^m\times\real^n$ to $\mcE_p$ for any $p>1$ with 
\begin{align}\label{eq:mcFtilde1}
\widetilde{\mcF}_1 u &\colon [0,1]\to \real^m,\notag\\
(\widetilde{\mcF}_1 u )(x)&\coloneqq \left(d_1(x)\cdot \left(\frac{d}{dx} u_1\right)(x),\dots ,d_m(x)\cdot \left(\frac{d}{dx} u_m\right)(x))\right)^{\top}.
\end{align}

To define the operator $\mcG$ we argue in analogy with \cite[Sec.~5]{KvN19}. First define 
\[
\mcH\coloneqq \mcE_2
\]
the product $L^2$-space, see \eqref{eq:mcE_2def}, which is a Hilbert space. We further define the multiplication operator $\Gamma\colon [0,T]\times \Omega\times\mcE^c\to\mathcal{L}(\mcH)$ as
\begin{equation}\label{eq:Gamma2def}
[\Gamma(t,\omega,U)Y](x)\coloneqq \begin{pmatrix}
	\Gamma_1(t,\omega, x,u(x)) & 0_{m\times n} \\ 
	0_{n\times m} & \Gamma_2(t,\omega, r)
\end{pmatrix}_{(m+n)\times(m+n)}\cdot Y
\end{equation}
for $U=\begin{pmatrix}
	u \\ r
\end{pmatrix}\in \mcE^c$ and $Y\in\mcH$ with 
\[\Gamma_1(t,\omega,x,u(x))=\begin{pmatrix} h_1(t,\omega, x,u_1(x)) & \cdots &  0\\
\vdots &  \ddots & \vdots \\
0 & \cdots & h_m(t,\omega,x,u_m(x))\end{pmatrix}_{m\times m}\]
and
\[\Gamma_2(t,\omega,r)=\begin{pmatrix} g_1(t,\omega,r_1) & \cdots &  0\\
\vdots &  \ddots & \vdots \\
0 & \cdots & g_n(t,\omega,r_n)\end{pmatrix}_{n\times n}.\]
Because of the assumptions on the functions $h_j$ and $g_i$, $\Gamma$ clearly maps into $\mathcal{L}(\mcH).$

Let $(\mcA_2,D(\mcA_2))$ be the generator on $\mcH=\mcE_2,$ see Proposition \ref{prop:sgrextend}, and pick $\kG\in(\frac{1}{4},\frac{1}{2})$. 
Using Lemma \ref{lemma:fractionalspaceiso}(2) we have that there is an isomorphism
\begin{equation}\label{eq:imath}
\imath\colon \mcE_2^{\kG} \to \left(\prod _{j=1}^m H_{0}^{2\kG}(0,1;\mu_j dx)\right)\times\real^n \eqqcolon\mcH_1.
\end{equation}
By Corollary \ref{cor:fractionalspaceincl}, $\mcH_1 \hookrightarrow\mcE^c$ holds. Using Corollary \ref{cor:fractionalspaceincl} again, we have that the there exists a continuous embedding
\begin{equation}\label{eq:jmath}
\jmath\colon \mcH_1 \to \mcE_p
\end{equation}
for $p\geq 2$ arbitrary.

Let $\nu>0$ arbitrary and define now $G$ by
\begin{equation}\label{eq:mcG2def}
(\nu-\mcA_p)^{-\kG}\mcG(t,\omega, U)Y\coloneqq \jmath\, \imath\,   (\nu-\mcA_2)^{-\kG}\Gamma(t,\omega,U)Y,\quad U\in \mcE^c,\; Y\in \mcH.
\end{equation}

\begin{lemma}\label{lem:mcG2prop}
\hspace{2em}
\begin{enumerate}[1.]
	\item Let $p>1$ arbitrary. Then the mapping defined in \eqref{eq:mcFtilde} can be extended to a linear and continuous operator from $\mcE^c$ into $\mcE_p^{-\frac{1}{2}}$, that we also call $\widetilde{\mcF}$.
	\item Let $p\geq 2$ and $\kG\in(\frac{1}{4},\frac{1}{2})$ be arbitrary. The operator $\mcG$ defined in \eqref{eq:mcG2def} maps $[0,T]\times\Omega\times \mcE^c$ into $\gamma (\mcH,\mcE_p^{-\kappa_{G}})$.
\end{enumerate}
\end{lemma}
\begin{proof}
\begin{enumerate}[1.]
	\item To prove the claim for $\widetilde{\mcF}$ let $p> 1$ arbitrary and $q\coloneqq (1-\frac{1}{p})^{-1}$. We first investigate the operator $\widetilde{\mcF}_1$ defined in \eqref{eq:mcFtilde1} and take $u\in (C^1[0,1])^m$, $v\in (W_0^{1,q}(0,1))^m$ obtaining that for a constant $c'>0$
\begin{align}
\left|\langle \widetilde{\mcF}_1 u, v\rangle\right|&=\left|\sum_{j=1}^m\int_0^1 d_j(x)u_j'(x)v_j(x)\dx\right|=\left|-\sum_{j=1}^m\int_0^1 u_j(x)(d_jv_j)'(x)\dx\right|\\
&\leq c'\cdot \|u\|_{(C[0,1])^m}\cdot \|d\|_{(W^{1,\infty}(0,1))^m}\cdot \|v\|_{(W_0^{1,q}(0,1))^m}
\end{align}
where $d=(d_1,\dots ,d_m)$ and $\langle \cdot,\cdot\rangle$ denotes the duality of $(W_0^{1,q}(0,1))^m$. Hence, for a positive constant $c$,
\[\frac{\left|\langle \widetilde{\mcF}_1u, v\rangle\right|}{\|v\|_{(W_0^{1,q}(0,1))^m}}\leq c\|u\|_{(C[0,1])^m}.\]
Since $(C^1[0,1])^m$ is dense in $(C[0,1])^m$, $\widetilde{\mcF}_1$ can be extended to a continuous linear operator 
\[\widetilde{\mcF}_1\colon (C[0,1])^m \to \left((W_0^{1,q}(0,1))^m\right)^*\] 
with
\begin{equation}\label{eq:mcFtildepf}
\|\widetilde{\mcF}_1 u\|_{\left((W_0^{1,q}(0,1))^m\right)^*}\leq c\|u\|_{(C[0,1])^m},\quad u\in (C[0,1])^m.
\end{equation}
By \cite[Thm.~3.1.4]{vNeervenbook} we have
\[\left(\mcE_q^{\frac{1}{2}}\right)^*\cong \mcE_p^{-\frac{1}{2}}\;\text{ for }\frac{1}{p}+\frac{1}{q}=1.\]
Lemma \ref{lemma:fractionalspaceiso} implies that
	\[\mcE_q^{\frac{1}{2}} \cong \left(\prod _{j=1}^m W_{0}^{1,q}(0,1;\mu_j dx)\right)\times \real^n\cong (W_0^{1,q}(0,1))^m\times \real^n\] 
holds, hence
\[\mcE_p^{-\frac{1}{2}}\cong \left(\mcE_q^{\frac{1}{2}}\right)^*\cong \left((W_0^{1,q}(0,1))^m\right)^*\times \real^n.\]
By the definition \eqref{eq:mcFtilde} of $\widetilde{\mcF}$ and by the extension \eqref{eq:mcFtildepf} of $\widetilde{\mcF}_1$ this means exactly that $\widetilde{\mcF}$ can be extended to a continuous linear operator from $\mcE^c$ in $\mcE_p^{-\frac{1}{2}}$ with
\begin{equation}\label{eq:Fhullamcontlin}
\|\widetilde{\mcF} U\|_{\mcE_p^{-\frac{1}{2}}}\leq c\|U\|_{\mcE^c},\quad U\in\mcE^c.
\end{equation}
	
	\item We can argue as in \cite[Sec.~10.2]{vNVW08}. Using \cite[Lem.~2.1(4)]{vNVW08}, we obtain in a similar way as in \cite[Cor.~2.2]{vNVW08}) that $\jmath \in\gamma (\mathcal{H}_1,\mcE_p)$, since $2\kG>\frac{1}{2}$ holds. Hence, by the definition of $\mcG$ and the ideal property of $\gamma$-radonifying operators, the mapping $\mcG$ takes values in $\gamma (\mcH,\mcE_p^{-\kappa_{G}})$.
\end{enumerate}
\end{proof}

The driving noise process $\mcW$ is defined by
\begin{equation}\label{eq:mcW2def}
\mcW(t)=\begin{pmatrix}
	w_1(t,\cdot)\\
	\vdots\\
	w_m(t,\cdot)\\
	\beta_1(t)\\
	\vdots\\
	\beta_n(t)
\end{pmatrix},
\end{equation}
and thus $(\mcW(t))_{t\in [0,T]}$ is a cylindrical Wiener process, defined on $(\Omega,\mathscr{F},\mathbb{P})$, in the Hilbert space $\mcH$ with respect to the filtration $\mathbb{F}$.
\medskip

Similar to \eqref{eq:V_T} for a fixed $T>0$ and $q\geq 1$ we define the space
\begin{equation}\label{eq:mcV_T}
\mcV_{T,q}\coloneqq L^q\left(\Omega;C((0,T];\mcE^c)\cap L^{\infty}(0,T;\mcE^c)\right)
\end{equation}
being a Banach space with norm
\begin{equation}\label{eq:mcV_Tnorm}
\left\|U\right\|^q_{\mcV_{T,q}}\coloneqq \mathbb{E}\sup_{t\in [0,T]}\|U(t)\|_{\mcE^c}^q,\quad U\in \mcV_{T,q},
\end{equation} 
This Banach space will play a crucial role for the solutions of \eqref{eq:SCPn}.

We will state now the result regarding system \eqref{eq:SCPn}.

\begin{theo}\label{theo:SCPnsolcont}
Let $\mcF$, $\widetilde{\mcF}$, $\mcG$ and $\mcW$ defined as in \eqref{eq:mcFdef}, \eqref{eq:mcFtilde}, \eqref{eq:mcG2def} and \eqref{eq:mcW2def}, respectively. 
Let $q>4$ be arbitrary.
Then for every $\xi\in L^q(\Omega,\mathscr{F}_0,\mathbb{P};\mcE^c)$ equation \eqref{eq:SCPn} has a unique global mild solution in $\mcV_{T,q}$.
\end{theo}

\begin{proof}
The condition $q > 4$ allows us to choose $2 \leq p < \infty$, $\theta\in (0,\frac{1}{2})$ and $\kG\in (\frac14 ,\frac12)$ such that
\begin{equation}\label{eq:thetap}
\theta>\frac{1}{2p}
\end{equation} 
and 
\begin{equation}\label{eq:thetakappaG}
0 < \theta+ \kG < \frac12 - \frac1q.
\end{equation} 
We will apply Theorem \ref{theo:KvN4.9gen} with $\theta$ and $\kG$ having the properties above. To this end we have to check Assumptions \ref{assum:mainvN4.9mod} for the mappings in \eqref{eq:SCPn}, taking $\mcA=\mcA_p$ for the $p$ chosen above. 

\begin{enumerate}[(a)]
	\item Assumption $(1)$ is satisfied because of the generator property of $\mcA_p$, see Proposition \ref{prop:sgrextend}.
	\item Assumption $(2)$ is satisfied since \eqref{eq:thetap} holds and we can use Corollary \ref{cor:fractionalspaceincl}.
	\item Assumption $(3)$ is satisfied because of Proposition \ref{prop:mcAonmcEc}.
	\item To satisfy Assumptions $(4)$ and $(5)$ we first remark that the locally Lipschitz continuity of $\mcF$ follows from \eqref{eq:fjdef}. In the following we have to consider vectors $U^*\in\partial\|U\|$ for $U=\left(\begin{smallmatrix}u\\ r\end{smallmatrix}\right)\in\mcE^c$. It is easy to see that there exists $U^*\in\partial\|U\|$ of the form
\[U^*=\left(\begin{smallmatrix}u^*\\ r^*\end{smallmatrix}\right)\]
with
\[u^*\in \partial\|u\|_{(C[0,1])^m}\text{ and }r^*\in\partial\|r\|_{\ell^{\infty}}.\]
Using that the functions $f_j$ are polynomials of the 4th variable (see \eqref{eq:fjdef}), a similar computation as in \cite[Ex.~4.2]{KvN12} shows that for all $j=1,\dots ,m$ and for a suitable constant $a'\geq 0$
\[f_j(t,\omega,x,\eta+\zeta)\cdot\sgn\eta\leq a'(1+|\zeta|^{2k_j+1})\]
holds. Using techniques from \cite[Sec.~4.3]{DPZ92} we obtain that
\[\langle \mcA U+\mcF (t,U+V), U^*\rangle\leq a'(1+\|V\|_{\mcE^c})^K+b'\|U\|_{\mcE^c}\]
with $K$ defined in \eqref{eq:defkK} and for all $U\in D(\mcA_p|_{\mcE^c})$, $V\in \mcE^c$ and $U^*\in\partial\|U\|.$
Following the computation of \cite[Ex.~4.5]{KvN12}, we obtain that for suitable positive constants $a,b,c$ and for all $(t,\omega,x)\in [0,T]\times\Omega\times [0,1]$ and $j=1,\dots ,m$
\[\left[f_j(t,\omega,x,\eta+\zeta)-f_j(t,\omega,x,\zeta)\right]\cdot\sgn\eta\leq a-b|\eta|^{2k_j+1}+c|\zeta|^{2k_j+1}\]
holds. Using again techniques from \cite[Sec.~4.3]{DPZ92} (see also \cite[Rem.~5.1.2 and (5.19)]{Ce03}, we obtain that for $k$ and $K$ defined in \eqref{eq:defkK} $K\geq k$ holds and
\[\langle \mcF (t,\omega, U+V)-\mcF (t,\omega, V), U^*\rangle\leq a''(1+\|V\|_{\mcE^c})^K-b''\|U\|_{\mcE^c}^k\]
for all $t\in[0,T]$, $\omega\in\Omega$, $U,V\in \mcE^c$ and $U^*\in\partial\|U\|.$ Furthermore,
\begin{equation}\label{eq:Ass3.5mcF}
\left\|\mcF(t,V)\right\|_{\mcE^c}\leq a''(1+\|V\|_{\mcE^c})^K
\end{equation}
for all $V\in \mcE^c.$
\item To check Assumption (6) we refer to Lemma \ref{lem:mcG2prop}. This implies that $\widetilde{F}\colon \mcE^c\to \mcE^{-\kF}$ with $\kF=\frac{1}{2}$. Since $\widetilde{F}$ is a continuous linear operator, the rest of the statement also follows.
\item To check Assumption (7) note that by Lemma \ref{lem:mcG2prop}, $\mcG$ takes values in $\gamma (\mcH,\mcE_p^{-\kappa_{G}})$ with $\mcH=\mcE_2$ and $\kG$ chosen above. We apply a similar computation as in the proof of \cite[Thm.~10.2]{vNVW08}. We fix $U,V\in \mcE^c$ and let
\[R\coloneqq \max\left\{\|U\|_{\mcE^c},\|V\|_{\mcE^c}\right\}.\]
Furthermore, we denote the matrix from \eqref{eq:Gamma2def} by
\[\mathcal{M}_{\Gamma}(t,\omega,U)\coloneqq \begin{pmatrix}
	\Gamma_1(t,\omega,\cdot,u(\cdot)) & 0_{m\times n} \\ 
	0_{n\times m} & \Gamma_2(t,\omega,r)
\end{pmatrix}_{(m+n)\times(m+n)},\quad\text{ for }U=\begin{pmatrix}
	u \\ r
\end{pmatrix}\in \mcE^c.\]
For $R>0$ we denote
\begin{equation}%\label{eq:}
L_g(R)\coloneqq\max_{1\leq i\leq n}L_{g_i}(R),\quad L_h(R)\coloneqq\max_{1\leq j\leq m}L_{h_j}(R),
\end{equation}
where the positive constants $L_{g_i}(R)$'s and $L_{h_j}(R)$'s are the corresponding Lipschitz constants of the functions $g_i$ and $h_j$, respectively, on the ball of radius $R$, see \eqref{eq:gidefLip} and \eqref{eq:hjdefLip}.
From the right-ideal property of the $\gamma$-radonifying operators and \eqref{eq:mcG2def} we have that
\begin{align}
&\left\|(-\mcA_p)^{-\kG}\left(\mcG(t,\omega,U)-\mcG(t,\omega,V)\right)\right\|_{\gamma (\mcH,\mcE_p)}\\
&\leq \left\|\jmath\, \imath\,   (-\mcA_2)^{-\kG}\right\|_{\gamma (\mcH,\mcE_p^{-\kappa_{G}})}\cdot\left\|\Gamma(t,\omega,U)-\Gamma(t,\omega,V)\right\|_{\mathcal{L}(\mcH)}\\
&\leq \left\|\jmath\, \imath\,   (-\mcA_2)^{-\kG}\right\|_{\gamma (\mcH,\mcE_p^{-\kappa_{G}})}\cdot\left\|\mathcal{M}_\Gamma(t,\omega,U)-\mathcal{M}_\Gamma(t,\omega,V)\right\|_{L^{\infty}\left(([0,1])^m\times\real^n,\real^{(m+n)\times(m+n)}\right)}\\
&\preceq \left\|\jmath\, \imath\,   (-\mcA_2)^{-\kG}\right\|_{\gamma (\mcH,\mcE_p^{-\kappa_{G}})}\cdot \max\{L_g(R),L_h(R)\}\cdot \left\|U-V\right\|_{\mcE^c}.
\end{align}
Hence, we obtain that $\mcG:[0,T]\times \mcE^c\to \gamma (\mcH,\mcE_p^{-\kappa_{G}})$ is locally Lipschitz continuous.\\
Using the assumptions \eqref{eq:gidefnov} and \eqref{eq:hjdefnov} on the functions $g_i$'s and $h_j$'s and an analogous computation as above, we obtain that $\mcG$ grows as required in Assumption $(7)$ as a map $[0,T]\times \mcE^c\to \gamma (\mcH,\mcE_p^{-\kappa_{G}})$.
\end{enumerate}
\end{proof}

In the following we treat the special case when $h_j\equiv 0$, $j=1,\dots ,m$, that is, there is stochastic noise only in the vertices of the network. To rewrite the equations \eqref{eq:stochsys2} in the form \eqref{eq:SCPn}, we define the operator $\mcG$ in a different way than it has been done in \eqref{eq:mcG2def}.

Instead of the operator in \eqref{eq:Gamma2def} we define $\Gamma\colon [0,T]\times\Omega\times\mcE^c\to\mathcal{L}(\mcE_p)$ as
\begin{equation}\label{eq:Gammadef}
\left[\Gamma(t,\omega,U)Y\right](x)\coloneqq \begin{pmatrix}
	0_{m\times m} & 0_{m\times n} \\ 
	0_{n\times m} & \Gamma_2(t,\omega,r)
\end{pmatrix}_{(m+n)\times(m+n)}\cdot Y
\end{equation}
for $U=\begin{pmatrix}
	u \\ r
\end{pmatrix}\in \mcE^c$ and $Y\in\mcE_p$ with 
\[\Gamma_2(t,\omega,r)=\begin{pmatrix} g_1(t,\omega,r_1) & \cdots &  0\\
\vdots &  \ddots & \vdots \\
0 & \cdots & g_n(t,\omega,r_n)\end{pmatrix}_{n\times n}.\]
Because of the assumptions on the functions $g_i$, $\Gamma$ clearly maps into $\mathcal{L}(\mcE_p).$

Now, let
\begin{equation}\label{eq:Rdef}
R\coloneqq \begin{pmatrix}
	0_{m\times m} & 0_{m\times n} \\ 
	0_{n\times m} & I_{n\times n}
\end{pmatrix}_{(m+n)\times(m+n)}
\end{equation} 
Then for all $p\geq 2$, $R\in\gamma(\mcH,\mcE_p)$ with $\mcH=\mcE_2$ holds since $R$ has finite dimensional range. 

Now $\mcG\colon [0,T]\times\Omega\times\mcE^c\to\gamma(\mcH,\mcE_p)$ will be defined as
\begin{equation}\label{eq:mcGdef}
\mcG(t,\omega,U)Y\coloneqq \Gamma(t,\omega,U)R Y,\quad U\in\mcE^c,\quad Y\in\mcH.
\end{equation}

In this case we obtain a better regurality in Theorem \ref{theo:SCPnsolcont}.

\begin{theo}\label{theo:SCPnsolconthj0}
Let $\mcF$, $\widetilde{\mcF}$, $\mcG$ and $\mcW$ defined as in \eqref{eq:mcFdef}, \eqref{eq:mcFtilde}, \eqref{eq:mcGdef} and \eqref{eq:mcW2def}, respectively, and assume that $h_j\equiv 0,$ $j=1,\dots m.$ 
Then for arbitrary $q>2$ and for every $\xi\in L^q(\Omega,\mathscr{F}_0,\mathbb{P};\mcE^c)$ equation \eqref{eq:SCPn} has a unique global mild solution in $\mcV_{T,q}$.
\end{theo}
\begin{proof}
We first chose $p\geq 2$ such that
\begin{equation}\label{eq:mainthmqp}
q\left(\frac{1}{2}-\frac{1}{2p}\right)>1.
\end{equation}
Hence, we can take $\theta$ satisfying
\[\frac{1}{2p}<\theta<\frac{1}{2}-\frac{1}{q}.\]
We will apply Theorem \ref{theo:KvN4.9gen} with $\theta$ having this property and $\kG=0$. To this end we have to check Assumptions \ref{assum:mainvN4.9mod} again for the mappings in \eqref{eq:SCPn}, taking $\mcA=\mcA_p$ for the $p$ chosen above. This can be done in the same way as in the proof of Theorem \ref{theo:SCPnsolcont} up to Assumption (7). It can be easily checked for $\kG=0$ for the operator $\mcG\colon [0,T]\times\Omega\times\mcE^c\to\gamma(\mcH,\mcE_p)$ defined in \eqref{eq:mcGdef}. If $U,V\in\mcE^c$ with $\|U\|_{\mcE^c},\|V\|_{\mcE^c}\leq r$ then
\begin{align*}
\left\|\mcG(t,\omega,U)-\mcG(t,\omega,V)\right\|_{\gamma(\mcH,\mcE_p)}&\leq \left\|\Gamma(t,\omega,U)-\Gamma(t,\omega,V)\right\|_{\mathcal{L}(\mcE_p)}\cdot \|R\|_{\gamma(\mcH,\mcE_p)}\\
&\leq \|R\|_{\gamma(\mcH,\mcE_p)}\cdot L_{\mcG}^{(r)}\cdot \|U-V\|_{\mcE^c},
\end{align*}
where $L_{\mcG}^{(r)}$ is the maximum of the Lipschitz-constants of the functions $g_i$ on the ball $\{x\in\real\colon |x|\leq r\}$ (see \eqref{eq:gidefLip}) and $\|R\|_{\gamma(\mcH,\mcE_p)}$ is finite.

Furthermore, applying \eqref{eq:gidefnov}, the last statement of Assumption $(7)$ follows similarly as above, hence there exists a constant $c'>0$ such that
\begin{equation}\label{eq:Ass3.5mcG}
\left\|\mcG(t,\omega,U)\right\|_{\gamma(\mcH,\mcE_p)}\leq c'\left(1+\|U\|_{\mcE^c}\right)^{\frac{k}{K}}
\end{equation}
for all $(t,\omega,U)\in [0,T]\times \Omega\times \mcE^c.$
\end{proof}

In the following theorem we will state a result regarding the regularity of the mild solution of \eqref{eq:SCPn} that exists according to Theorem \ref{theo:SCPnsolcont}. We will show that the trajectories of the solutions are actually continuous in the vertices of the graph, hence they lie in the space
\begin{equation}\label{eq:mcB}
\mcB\coloneqq\left\{\left(\begin{smallmatrix} u\\r\end{smallmatrix}\right)\in  D(L)\times \real^n: L u=r\right\},
\end{equation} 
where $(L,D(L))$ is the boundary operator defined in \eqref{eq:Ldef}. The space $\mcB$ can be looked at as the Banach space of all continuous functions on the graph $\mathsf{G}$ with norm
\begin{equation}\label{eq:normmcB}
\|U\|_{\mcB}=\max_{j=1,\dots ,m}\sup_{[0,1]}|u_j|=\|U\|_{\mcE^c},\quad U=\left(\begin{smallmatrix} u\\r\end{smallmatrix}\right)\in \mcB.
\end{equation} 
It is easy to see that
\begin{equation}\label{eq:mcBEc}
\mcB\cong D(L)\text{ and }\mcB\subset\mcE^c.
\end{equation}

We can again prove the following continuous embeddings. In contrast to Corollary \ref{cor:fractionalspaceincl}, also the first embedding will be dense.

\begin{prop}\label{prop:fractionalspaceinclmcB}
Let $\mcE_p^{\theta}$ defined in \eqref{eq:fractdommcE} for $p\in[1,\infty)$. Then for $\theta> \frac{1}{2p}$ the following continuous, dense embeddings are satisfied:
\begin{equation}\label{eq:EthetainclmcB}
\mcE_p^{\theta}\hookrightarrow\mcB\hookrightarrow \mcE_p.
\end{equation} 
\end{prop}
\begin{proof}
According to Lemma \ref{lemma:fractionalspaceiso}$(2)$ we have that for $\theta> \frac{1}{2p}$
\begin{equation}\label{eq:interpolpf}
\mcE_p^{\theta} \cong \left(\prod _{j=1}^m W_{0}^{2\theta,p}(0,1;\mu_j dx)\right)\times \real^n
\end{equation}
holds.
In \cite[Lem.~3.6]{KS21} we have proved that
\begin{equation}\label{eq:mcBszorzat}
\mcB\cong (C_0[0,1])^m\times\real^n,
\end{equation}
where $C_0[0,1]$ denotes such continuous functions that are $0$ at the endpoints of the interval $[0,1]$. Hence, combining this with \eqref{eq:interpolpf} and using Sobolev imbedding, we obtain that for $\theta>\frac{1}{2p}$ the continuous, dense embedding
\[\mcE_p^{\theta}\hookrightarrow \mcB\]
is satisfied. Using \eqref{eq:mcBszorzat} again, we have $\mcB\hookrightarrow\mcE_p$, and the claim follows.
\end{proof}

To the analogy of $\mcV_{T,q}$, we define for a fixed $T>0$ and $q\geq 1$
\begin{equation}\label{eq:V_Thullam}
\widetilde{\mcV}_{T,q}\coloneqq L^q\left(\Omega;C((0,T];\mcB)\cap L^{\infty}(0,T;\mcB)\right)
\end{equation}
being a Banach space with norm
\begin{equation}\label{eq:V_Thullamnorm}
\left\|U\right\|^q_{\widetilde{\mcV}_{T,q}}\coloneqq \mathbb{E}\sup_{t\in [0,T]}\|U(t)\|_{\mcB}^q,\quad U\in \widetilde{\mcV}_{T,q},
\end{equation}

In the following we will show that the trajectories of the solution of \eqref{eq:SCPn} lie in $\mcB$.

\begin{theo}\label{theo:Holderreg}
Let $\mcF$, $\widetilde{\mcF}$, $\mcG$ and $\mcW$ defined as in \eqref{eq:mcFdef}, \eqref{eq:mcFtilde}, \eqref{eq:mcG2def} and \eqref{eq:mcW2def}, respectively. Let $q>4$ be arbitrary. Then for every $\xi\in L^{q}(\Omega,\mathscr{F}_0,\mathbb{P};\mcE^c)$ equation \eqref{eq:SCPn} has a unique global mild solution in $\widetilde{\mcV}_{T,q}$.
\end{theo}

\begin{proof}
By Theorem \ref{theo:SCPnsolcont} there exists a global mild solution $\mcX\in \mcV_{T,q}$, that is
\begin{equation}\label{eq:Holderregpf1}
\mcX\in L^{q}(\Omega; C((0,T];\mcE^c)\cap L^{\infty}(0,T;\mcE^c)).
\end{equation}
This solution satisfies the following implicit equation (see \eqref{eq:mildsol}):
\begin{equation}\label{eq:proofmildsol}
\mcX(t)=\mcS(t)\xi+\mcS\ast \mcF(\cdot,\mcX(\cdot))(t)+\mcS\ast \widetilde{\mcF}(\mcX(\cdot))(t)+\mcS\diamond \mcG(\cdot,\mcX(\cdot))(t),
\end{equation}
where $\mcS$ denotes the semigroup generated by $\mcA_p$ on $\mcE_p$ for some $p\geq 2$ large enough, $\ast$ denotes the usual convolution, $\diamond$ denotes the stochastic convolution with respect to $\mcW.$ We only have to show that for almost all $\omega\in\Omega$ for the trajectories 
\begin{equation}\label{eq:Holderregpf2}
\mcX(\cdot)\in C((0,T];\mcB)\cap L^{\infty}(0,T;\mcB)
\end{equation} 
holds. Then $\mcX\in \widetilde{\mcV}_{T,q}$ is satisfied since the norms on $\mcE^c$ and $\mcB$ coincide and \eqref{eq:Holderregpf1} is true. We will show \eqref{eq:Holderregpf2} by showing it for all the three terms on the right-hand-side of \eqref{eq:proofmildsol}.

We first fix $0<\alpha<\frac{1}{2}$ and $\eta> 0$ such that 
\[\eta+\frac{1}{4}<\alpha-\frac{1}{q}\] 
holds. It is possible beacuse of the assumption $q>4$. Then we choose $\kG\in(\frac{1}{4},\frac{1}{2})$ such that
\begin{equation}\label{eq:HolderregpfetakG}
\eta+\kG<\alpha-\frac{1}{q}
\end{equation}
is satisfied. We further fix $p\geq 2$ such that
\begin{equation}\label{eq:Holderregpfp}
\frac{1}{2p}<\eta.
\end{equation}
\begin{enumerate}
	\item  Using \eqref{eq:mcBEc}, we have that almost surely $\xi\in \mcE^c$ holds. Hence, using the analiticity of $\mcS$ on $\mcE^c$ (see Proposition \ref{prop:mcAonmcEc}) and the obvious fact $D(\mcA)\subset\mcB$ (see \eqref{eq:mcAp}), we have that almost surely 
\begin{equation}\label{eq:Sxi}
\mcS(\cdot)\xi\in C((0,T];\mcB)\cap L^{\infty}(0,T;\mcB)
\end{equation} 
also holds.
\item For the deterministic convolution term with $\mcF$ in \eqref{eq:proofmildsol} observe that by the choice of the constants
%\begin{equation}\label{eq:mainthmBp}
%\frac{1}{2p}<1-\frac{1}{q}
%\end{equation}
%holds (it is possible since $q>2$).  
\begin{equation}\label{eq:peta}
\frac{1}{2p}<\eta<1-\frac{1}{q}
\end{equation}
holds. We now apply \cite[Lem.~3.6]{vNVW08} with $\alpha=1$, $\theta=\lambda=0$, $q$ instead of $p$ and for $\eta$. Hence, we obtain that there exist constants $C\geq 0$ and $\ve>0$ such that
\begin{equation}\label{eq:est2nd1}
\|\mcS\ast \mcF(\cdot,\mcX(\cdot))(t)\|_{C([0,T];\mcE_p^{\eta})}\leq CT^{\ve}\|\mcF(\cdot,\mcX(\cdot))\|_{L^q(0,T;\mcE_p)}.
\end{equation}
Taking the $q$th power on the right-hand-side of \eqref{eq:est2nd1}, using Corollary \ref{cor:fractionalspaceincl} and \eqref{eq:Ass3.5mcF} we obtain that
\begin{align}
\|\mcF(\cdot,\mcX(\cdot))\|^q_{L^q(0,T;\mcE_p)}&=\int_0^T\|\mcF(s,\mcX(s))\|^q_{\mcE_p}\ds\\
& \lesssim\int_0^T\|\mcF(s,\mcX(s))\|^q_{\mcE^c}\ds\\
&\lesssim \int_0^T(1+\|\mcX(s)\|^{K\cdot q}_{\mcE^c})\ds\\
&\lesssim (1+\sup_{t\in[0,T]}\|\mcX(t)\|^{K\cdot q}_{\mcE^c}).
\end{align}
Hence,
\begin{equation}%\label{eq:SastF}
\|\mcS\ast \mcF(\cdot,\mcX(\cdot))\|^q_{C([0,T];\mcE_p^{\eta})}
\leq C_T\cdot \left(1+\sup_{t\in[0,T]}\|\mcX(t)\|_{\mcE^c}^{K\cdot q}\right).
\end{equation}
By Proposition \ref{prop:fractionalspaceinclmcB} and \eqref{eq:peta} we obtain that $\mcS\ast \mcF(\cdot,\mcX(\cdot)) \in C([0,T];\mcB)$ holds and for a positive constant $C_T'>0$
\begin{equation}\label{eq:SastF}
\|\mcS\ast \mcF(\cdot,\mcX(\cdot))\|^q_{C([0,T];\mcB)}
\leq C_T'\cdot \left(1+\sup_{t\in[0,T]}\|\mcX(t)\|_{\mcE^c}^{K\cdot q}\right).
\end{equation}
Since we know by \eqref{eq:Holderregpf1} that for almost all $\omega\in\Omega$ the right-hand-side is finite, we obtain that the left-hand-side is almost surely finite.
\item For the deterministic convolution term with $\widetilde{\mcF}$ in \eqref{eq:proofmildsol} we proceed similarly as before.
We apply \cite[Lem.~3.6]{vNVW08} with $\alpha=1$, $\lambda=0$, $\theta=\frac{1}{2}$, $q$ instead of $p$ and for $\eta$. 
We obtain that there exist constants $C\geq 0$ and $\ve>0$ such that
\begin{equation}\label{eq:est2nd1Fhullam}
\|\mcS\ast \widetilde{\mcF}(\mcX(\cdot))(t)\|_{C([0,T];\mcE_p^{\eta})}\leq CT^{\ve}\|\widetilde{\mcF}(\mcX(\cdot))\|_{L^q(0,T;\mcE_p^{-\frac{1}{2}})}.
\end{equation}
Taking the $q$th power on the right-hand-side of \eqref{eq:est2nd1Fhullam} and using \eqref{eq:Fhullamcontlin} we obtain that
\begin{equation}
\|\widetilde{\mcF}(\mcX(\cdot))\|^q_{L^q(0,T;\mcE_p^{-\frac{1}{2}})}=\int_0^T\|\widetilde{\mcF}(\mcX(s))\|^q_{\mcE_p^{-\frac{1}{2}}}\ds \lesssim \int_0^T\|\mcX(s)\|^{q}_{\mcE^c}\ds \lesssim \sup_{t\in[0,T]}\|\mcX(t)\|^{q}_{\mcE^c}.
\end{equation}
Hence,
\begin{equation}%\label{eq:SastF}
\|\mcS\ast \widetilde{\mcF}(\mcX(\cdot))\|^q_{C([0,T];\mcE_p^{\eta})}
\leq C_T\cdot \sup_{t\in[0,T]}\|\mcX(t)\|_{\mcE^c}^{q}.
\end{equation}
By Proposition \ref{prop:fractionalspaceinclmcB} and \eqref{eq:peta} we obtain that $\mcS\ast \widetilde{\mcF}(\mcX(\cdot)) \in C([0,T];\mcB)$ holds and for a positive constant $C_T'>0$
\begin{equation}\label{eq:SastFhullam}
\|\mcS\ast \widetilde{\mcF}(\mcX(\cdot))\|^q_{C([0,T];\mcB)}
\leq C_T'\cdot \sup_{t\in[0,T]}\|\mcX(t)\|_{\mcE^c}^{q}.
\end{equation}
Since we know by \eqref{eq:Holderregpf1} that for almost all $\omega\in\Omega$ the right-hand-side is finite, we obtain that the left-hand-side is almost surely finite.
\item We now prove that for the stochastic convolution term $\mcS\diamond \mcG(\cdot,\mcX(\cdot))\in C([0,T];\mcB)$ almost surely holds by showing that
\begin{equation}\label{eq:Holderreg2pf1}
\mathbb{E}\left\|\mcS\diamond \mcG(\cdot,\mcX(\cdot))\right\|^q_{C([0,T];\mcB)}
\end{equation}
is finite. By \eqref{eq:HolderregpfetakG} we can apply \cite[Prop.~4.2]{vNVW08} with $\la=0$, $\theta=\kG$, $q$ instead of $p$, $\alpha$ and $\eta$, and we have that there exist $\ve>0$ and $C\geq 0$ such that
\begin{align}
&\mathbb{E}\left\|\mcS\diamond \mcG(\cdot,\mcX(\cdot))\right\|^q_{C([0,T],\mcE_p^{\eta})}\\ &\leq C^qT^{\ve q}\int_0^T\mathbb{E}\left\|s\mapsto (t-s)^{-\alpha}\mcG(s,\mcX(s))\right\|^q_{\gamma(L^2(0,t;\mcH),\mcE_p^{-\kG})}\dt.
\end{align}
In the following we proceed similarly as done in the proof of \cite[Thm.~4.3]{KvN12}, with $N=1$ and $q$ instead of $p.$ Since $\mcE_p^{-\kG}$ is a Banach space of type $2$ (because $\mcE_p$ is of that type), the continuous embedding
\begin{equation}\label{eq:L2embedding}
L^2(0,t;\gamma(\mcH,\mcE_p^{-\kG}))\hookrightarrow \gamma(L^2(0,t;\mcH),\mcE_p^{-\kG})
\end{equation}
holds.
Using this, Young's inequality and the growth property of $\mcG$ (see the proof of Theorem \ref{theo:SCPnsolcont}), respectively, we obtain the following estimates
\begin{align}
&\mathbb{E}\left\|\mcS\diamond \mcG(\cdot,\mcX(\cdot))\right\|^q_{C([0,T],\mcE_p^{\eta})}\\
&\quad\lesssim T^{\ve q}\int_0^T\mathbb{E}\left\|s\mapsto (t-s)^{-\alpha}\mcG(s,\mcX(s))\right\|^q_{L^2(0,t;\gamma(\mcH,\mcE_p^{-\kG}))}\dt\\
&\quad=T^{\ve q}\mathbb{E}\int_0^T\left(\int_0^t (t-s)^{-2\alpha}\left\|\mcG(s,\mcX(s))\right\|^2_{\gamma(\mcH,\mcE_p^{-\kG})}\ds\right)^{\frac{q}{2}}\dt\\
&\quad\leq T^{\ve q}\left(\int_0^T t^{-2\alpha}\dt\right)^{\frac{q}{2}}\mathbb{E}\int_0^T \left\|\mcG(t,\mcX(t))\right\|^q_{\gamma(\mcH,\mcE_p^{-\kG})}\dt \\
&\quad\lesssim T^{(\frac{1}{2}-\alpha+\ve)q} \mathbb{E}\int_0^T\left(1+\|\mcX(t)\|_{\mcE^c}\right)^{\frac{k}{K}q}\dt\\
&\quad \lesssim T^{(\frac{1}{2}-\alpha+\ve)q+1}\left(1+\|\mcX\|^{\frac{k}{K}q}_{\mcV_{T,q}}\right).
\end{align}
Hence, for each $T>0$ there exists constant $C_{T}>0$ such that
\begin{equation}\label{eq:SdiamondG2}
\mathbb{E}\left\|\mcS\diamond \mcG(\cdot,\mcX(\cdot))\right\|^q_{C([0,T],\mcE_p^{\eta})}
\leq C'_{T}\cdot\left(1+\|\mcX\|^{\frac{k}{K}q}_{\mcV_{T,q}}\right)
\end{equation}
Using that $\frac{k}{K}<1$, we have that
\begin{equation}%\label{eq:SdiamondG}
\mathbb{E}\left\|\mcS\diamond \mcG(\cdot,\mcX(\cdot))\right\|^q_{C([0,T];\mcE_p^{\eta})}
\leq C''_{T}\cdot\left(1+\|\mcX\|^{q}_{\mcV_{T,q}}\right)
\end{equation}
holds. By Proposition \ref{prop:fractionalspaceinclmcB} and \eqref{eq:Holderregpfp} we obtain that for a positive constant $\tilde{C}_T>0$
\begin{equation}\label{eq:SdiamondG}
\mathbb{E}\left\|\mcS\diamond \mcG(\cdot,\mcX(\cdot))\right\|^q_{C([0,T];\mcB)}
\leq \tilde{C}_T\cdot\left(1+\|\mcX\|^{q}_{\mcV_{T,q}}\right)
\end{equation}
is satisfied.
\end{enumerate} 
Finally, by \eqref{eq:Sxi}, \eqref{eq:SastF} and \eqref{eq:SdiamondG}, we obtain \eqref{eq:Holderregpf2} and hence the proof is complete.
\end{proof}

We again treat the case when $h_j\equiv 0$, $j=1,\dots ,m$ separately by defining the operator $\mcG$ as in \eqref{eq:mcGdef} to obtain better regurality for the solutions.

\begin{theo}\label{theo:Holderreghj0}
Let $\mcF$, $\widetilde{\mcF}$, $\mcG$ and $\mcW$ defined as in \eqref{eq:mcFdef}, \eqref{eq:mcFtilde}, \eqref{eq:mcGdef} and \eqref{eq:mcW2def}, respectively, and assume that $h_j\equiv 0,$ $j=1,\dots m.$ 
Then for arbitrary $q>2$ and for every $\xi\in L^{q}(\Omega,\mathscr{F}_0,\mathbb{P};\mcE^c)$ equation \eqref{eq:SCPn} has a unique global mild solution in $\widetilde{\mcV}_{T,q}$.
\end{theo}
\begin{proof}
The claim can be proved analogously to Theorem \ref{theo:Holderreg} except for step (4). To show that in this case $\mcS\diamond \mcG(\cdot,\mcX(\cdot))\in C([0,T];\mcB)$ almost surely holds we first fix $0<\alpha<\frac{1}{2}$ and $p\geq 2$ such that
\begin{equation}\label{eq:mainthmBp2}
\frac{1}{2p}<\alpha-\frac{1}{q}
\end{equation} 
holds (it is possible since $q>2$). We further choose $\eta>0$ such that 
\begin{equation}\label{eq:peta2}
\frac{1}{2p}<\eta<\alpha-\frac{1}{q}
\end{equation} 
is satisfied. Applying \cite[Prop.~4.2]{vNVW08} with $\theta=\la=0$ and $q$ instead of $p$, we have that there exist $\ve>0$ and $C\geq 0$ such that
\begin{equation}
\mathbb{E}\left\|\mcS\diamond \mcG(\cdot,\mcX(\cdot))\right\|^q_{C([0,T];\mcE_p^{\eta})}\leq C^qT^{\ve q}\int_0^T\mathbb{E}\left\|s\mapsto (t-s)^{-\alpha}\mcG(s,\mcX(s))\right\|^q_{\gamma(L^2(0,t;\mcH),\mcE_p)}\dt.
\end{equation}
In the following we proceed similarly as done in the proof of \cite[Thm.~4.3]{KvN12}, with $N=1$ and $q$ instead of $p.$ Since $\mcE_p$ is a Banach space of type $2$, we can use the continuous embedding
\[ L^2(0,t;\gamma(\mcH,\mcE_p))\hookrightarrow \gamma(L^2(0,t;\mcH),\mcE_p),\]
Young's inequality and \eqref{eq:Ass3.5mcG}, respectively, to obtain the following estimates
\begin{align}
\mathbb{E}\left\|\mcS\diamond \mcG(\cdot,\mcX(\cdot))\right\|^q_{C([0,T];\mcE_p^{\eta})}&\lesssim T^{\ve q}\int_0^T\mathbb{E}\left\|s\mapsto (t-s)^{-\alpha}\mcG(s,\mcX(s))\right\|^q_{L^2(0,t;\gamma(\mcH,\mcE_p))}\dt\\
&=T^{\ve q}\mathbb{E}\int_0^T\left(\int_0^t (t-s)^{-2\alpha}\left\|\mcG(s,\mcX(s))\right\|^2_{\gamma(\mcH,\mcE_p)}\ds\right)^{\frac{q}{2}}\dt\\
&\leq T^{\ve q}\left(\int_0^T t^{-2\alpha}\dt\right)^{\frac{q}{2}}\mathbb{E}\int_0^T \left\|\mcG(t,\mcX(t))\right\|^q_{\gamma(\mcH,\mcE_p)}\dt \\
&\lesssim T^{(\frac{1}{2}-\alpha+\ve)q} \mathbb{E}\int_0^T\left(1+\|\mcX(t)\|_{\mcE^c}\right)^{\frac{k}{K}q}\dt\\
&\lesssim T^{(\frac{1}{2}-\alpha+\ve)q+1}\left(1+\mathbb{E}\sup_{t\in[0,T]}\|\mcX(t)\|^{\frac{k}{K}q}_{\mcE^c}\right).
\end{align}
Using that $\frac{k}{K}<1$, we have that there exists a constant $C''_{T}$ such that
\begin{equation}%\label{eq:SdiamondG}
\mathbb{E}\left\|\mcS\diamond \mcG(\cdot,\mcX(\cdot))\right\|^q_{C([0,T];\mcE_p^{\eta})}
\leq C''_{T}\cdot\left(1+\|\mcX\|^{q}_{\mcV_{T,q}}\right)
\end{equation}
holds. By Proposition \ref{prop:fractionalspaceinclmcB} and \eqref{eq:peta2} we obtain that for a positive constant $\tilde{C}_T>0$
\begin{equation*}
\mathbb{E}\left\|\mcS\diamond \mcG(\cdot,\mcX(\cdot))\right\|^q_{C([0,T];\mcB)}
\leq \tilde{C}_T\cdot\left(1+\|\mcX\|^{q}_{\mcV_{T,q}}\right)
\end{equation*}
is satisfied and thus
\begin{equation}%\label{eq:SdiamondG}
\mcS\diamond \mcG(\cdot,\mcX(\cdot))\in C([0,T];\mcB)
\end{equation} 
almost surely.
\end{proof}

\begin{rem}
If, in Theorems \ref{theo:Holderreg} and \ref{theo:Holderreghj0}, the initial condition satisfies $\xi\in L^{q}(\Omega,\mathscr{F}_0,\mathbb{P};\mcB)$, then the global mild solution belongs even to $L^q(\Omega;C([0,T];\mcB))$ instead of $\widetilde{\mcV}_{T,q}$. This follows from the fact that the semigroup $\mcS$ is strongly continuous on $\mcB$ which can be shown by a completely analogous argument as in \cite[Prop.~3.8]{KS21arxiv}.
\end{rem}

%%%%%%%%%%%%%%%%%%%%%%%%%%%%%%%%% APPENDIX %%%%%%%%%%%%%%%%%%%%%%%%%%%%%%%%%%%%%%%

\appendix
\section{Proof of Proposition \ref{prop:formA}}\label{app:pfpropdetermL2}
We defined the form $\ea$ in \eqref{eq:form} -- \eqref{eq:domform} as
\begin{align}
\ea(U,V)&\coloneqq \sum_{j=1}^m\int_0^1 \mu_j c_j(x) u'_j(x) \overline{g'_j(x)} dx+\sum_{j=1}^m\int_0^1 \mu_j p_j(x)u_j(x) \overline{v_j(x)} dx-\langle Mr,q\rangle_{\comp^n}\\
&\text{ for }U=\left(\begin{smallmatrix}u\\ r\end{smallmatrix}\right),\, V=\Big(\begin{smallmatrix}g\\ q\end{smallmatrix}\Big)
\end{align}
on the Hilbert space $\mathcal{E}_2$ with domain
\begin{equation}
D\left(\ea\right)=\mathcal{V}\coloneqq \left\{U=\left(\begin{smallmatrix}u\\ r\end{smallmatrix}\right)\colon u\in \left(H^1(0,1)\right)^m\cap D(L),\, r\in \mathbb{C}^{n},\; L u=r\right\}.
\end{equation}
The operator associated to a form is defined in Definition \ref{defi:opassform} as
\begin{equation*}
\begin{array}{rcl}
D(\mcB)&\coloneqq&\left\{U\in \mcV:\exists V\in \mcE_2 \hbox{ s.t. } \ea(U,H)=\langle V,H\rangle_{\mcE_2}\; \forall H\in \mcV\right\},\\
\mcB U&\coloneqq&-V.
\end{array}\end{equation*}
The operator $\left(\mathcal{A}_2,D(\mathcal{A}_2)\right)$ from \eqref{eq:calA}--\eqref{eq:domcalA} is defined on $\mcE_2$ as
\begin{equation}
\mcA_2=\begin{pmatrix}
	A_{max} & 0\\
	C & M
\end{pmatrix}
\end{equation}
with domain
\begin{equation}
D(\mcA_2)=\left\{ \left(\begin{smallmatrix}u\\ r\end{smallmatrix}\right)\in D(A_{max})\times \mathbb{C}^{n}\colon L u=r\right\}.
\end{equation}

\begin{prop}\label{prop:formAapp}
The operator associated with the form $\ea$ is $\left(\mathcal{A}_2,D(\mathcal{A}_2)\right)$.
\end{prop}
\begin{proof}
Denote by $\left(\mathcal{B},D(\mathcal{B})\right)$ the operator associated with $\ea$.
Let us first show that $\mcA_2\subset \mcB$. Take $U=\left(\begin{smallmatrix}u\\ r\end{smallmatrix}\right)\in D(\mcA_2)$. Then for all
$H=\Big(\begin{smallmatrix}h\\ d\end{smallmatrix}\Big)\in \mcV$
\begin{align}\label{parts}
\ea(U,H)&=\sum_{j=1}^m \int_0^1 \mu_j c_j(x) u'_j(x)\overline{h'_j(x)}dx-\langle Mr,d\rangle_{\comp^n}\notag\\
&= \sum_{j=1}^m\left[\mu_j c_j u'_j\overline{h_j}\right]_0^1 - \sum_{j=1}^m \int_0^1 \mu_j (c_j u_j')'(x)\overline{h_j(x)}dx\\
&-\langle Mr,d\rangle_{\comp^n}.\notag
\end{align}
Using the entries \eqref{eq:fiijpm} of the incidence matrix $\Phi$, the first term
above can be written as
\[
\sum_{j=1}^m\left[\mu_j c_j u'_j\overline{h_j}\right]_0^1 = \sum_{j=1}^m\sum_{i=1}^n \mu_j c_j(\mv_i)(\phi_{ij}^- -\phi_{ij}^+)u'_j(\mv _i)\overline{h_j(\mv _i)}.
\]
Observe now that by the definition of $\mcV$
\[L h= d\]
holds. Changing the sums, and using \eqref{eq:matrixFiw} and \eqref{eq:opC}, we have that
\[\sum_{j=1}^m\left[\mu_j c_j u'_j\overline{h_j}\right]_0^1 =\sum_{i=1}^n \overline{d_i} \underbrace{\sum_{j=1}^m (\omega_{ij}^- -\omega_{ij}^+)u'_j(\mv _i)}_{=-[Cu]_i} =-\langle Cu,d\rangle_{\comp^n}.
\]
The second term in \eqref{parts} is by the definition \eqref{eq:opAmax}--\eqref{eq:domAmax} of $(A_{max},D(A_{max}))$
\[
-\sum_{j=1}^m \int_0^1 \mu_j (c_j u_j')'(x)\overline{h_j(x)}dx=-\langle A_{max} u,h\rangle_{E_2},
\]
which makes sense because $A_{max}f\in E_2$. 
Hence,
\begin{align*}
\ea(U,H)&= -\langle A_{max}u,h\rangle_{E_2}-\langle Cu,d\rangle_{\comp^n}-\langle Mr,d\rangle_{\comp^n}\\
&=-\langle \mcA_2 U, H\rangle_{\mcE_2}
\end{align*}
The proof of the inclusion $\mcA_2\subset \mcB$ is completed.\\

To check the converse inclusion $\mcB\subset \mcA_2$ take $U=\left(\begin{smallmatrix}u\\ r\end{smallmatrix}\right)\in D(\mcB)$. By definition, there exists $V=\Big(\begin{smallmatrix}v\\ q\end{smallmatrix}\Big)\in \mcE_2$ such that
\[\ea(U,H)=\langle V,H\rangle_{\mcE_2}\text{ for all } H=\Big(\begin{smallmatrix}h\\ d\end{smallmatrix}\Big)\in \mcV,\]
and $\mcB U=-V$. In particular,
\begin{equation}\label{integr}
\sum_{j=1}^m \int_0^1 \mu_j c_j(x)u'_j(x)\overline{h'_j(x)}dx - \langle Mr,d\rangle_{\comp^n}
=\sum_{j=1}^m \int_0^1 v_j(x)\overline{h_j(x)}\mu_j dx+ \langle q,d\rangle_{\comp^n}
\end{equation}
for all $H=\Big(\begin{smallmatrix}h\\ d\end{smallmatrix}\Big)\in \mcV$, hence also for all $h^j$ of the form
\begin{equation*}
h^j=\left(\begin{smallmatrix}
0\\
\vdots\\
h_j\\
\vdots\\
0
\end{smallmatrix}\right)\leftarrow j^{\rm th}\hbox{ row},\;\; h_j\in H^1_0(0,1)
\end{equation*}
and $L h^j=d=0_{\comp^n}.$
From this follows that~\eqref{integr} in fact implies
\[\int_0^1 \mu_j c_j(x)u'_j(x)\overline{h'_j(x)}dx = \int_0^1 v_j(x)\overline{h_j(x)}\mu_j dx\hbox{\; for all } j=1,\ldots,m,\;\; h_j\in H^1_0(0,1).\]
By definition of weak derivative this means that $c_j\cdot u_j'\in
H^1(0,1)$ for all $j=1,\ldots,m$. 
Since $0<c_j\in H^1(0,1)$, it follows that in fact $u_j'\in H^1(0,1)$ for all
$j=1,\ldots,m$. We conclude that $u\in \left(H^2(0,1)\right)^m$, hence by $U\in\mcV$, also $U\in D(\mcA_2)$ holds.

Moreover, integrating by parts as in~\eqref{parts} we see -- analogously to the first part of the proof -- that
if~\eqref{integr} holds for some $H=\Big(\begin{smallmatrix}h\\ d\end{smallmatrix}\Big)\in \mcV$, then
\begin{align*}&- \sum_{j=1}^m \int_0^1 \mu_j (c_j u_j')'(x)\overline{h_j(x)}dx-\langle Cu,d\rangle_{\comp^n}-\langle Mr,d\rangle_{\comp^n}\\
&=\sum_{j=1}^m \int_0^1 v_j(x)\overline{h_j(x)}\mu_j dx+\langle q,d\rangle_{\comp^n}.
\end{align*}  
That is 
\[\ea(U,H)=-\langle \mcA_2 U, H\rangle_{\mcE_2}=\langle V, H\rangle_{\mcE_2}\]
for arbitrary $H\in\mcV$, hence $ \mcA_2 U=-V=\mcB U$ and this completes the proof.
\end{proof}

\section{Proof of Proposition \ref{prop:sgrextend}}\label{app:extsgrLp}

We defined the spaces $\mcE_p$ for $p\in[1,\infty]$ in \eqref{eq:mcEp}. In the subsequent lemma we prove crucial properties of the semigroup $(\mcT_2(t))_{t\geq 0}$ that will imply its extendability to the spaces $\mcE_p$. The proof is similar to that of \cite[Lem.~4.1 and Prop.~5.3]{MR07} except the fact that we have non-diagonal matrix $M$. Therefore we give it in details. 

\begin{lemma}\label{lem:subMarkov}
If Assumption \ref{as:M} holds for $M$, then the semigroup $(\mcT_2(t))_{t\geq 0}$ on $\mcE_2$, associated with $\ea$,
is \emph{sub-Markovian}, i.e., it is real, positive, and contractive on $\mcE_\infty$.
\end{lemma}

\begin{proof}
By~\cite[Prop.~2.5, Thm.~2.7, and Cor.~2.17]{Ou05}, we need to
check that the following criteria are verified for the domain $\mcV$
of $\ea$, see \eqref{eq:domform}:
\begin{itemize}
\item ${U}\in \mcV \Rightarrow \overline{U}\in \mcV \hbox{ and } {\ea}({\rm Re}\,{U},{\rm Im}\,{U})\in\mathbb{R}$,
\item ${U}\in \mcV, U\hbox{ real-valued }\Rightarrow \vert {U}\vert \in \mcV \hbox{ and } \ea(\vert U\vert,\vert U\vert)\leq \ea(U,U)$,
\item $0\leq U\in \mcV \Rightarrow 1\wedge {U}\in \mcV \hbox{ and } \ea(1\wedge U,(U-1)^+)\geq 0$.
\end{itemize}

It is clear that $\overline{k}\in H^1(0,1)$ if $k\in H^1(0,1)$. Further, if $k$ is real valued, then $\vert
k\vert\in H^1(0,1)$ and $\vert k\vert'=\sgn k\cdot k'$, and if $0\leq k$, then $1\wedge k\in H^1(0,1)$ with
$(1\wedge k)'=k'\cdot {\mathbf{1}}_{\{k<1\}}$ and $((k-1)^{+})'=k'\cdot {\mathbf{1}}_{\{k>1\}}$.

Take any 
\[U=\left(\begin{smallmatrix} u\\r\end{smallmatrix}\right)\in \mcV,\quad u\in \left(H^1(0,1)\right)^m\cap D(L),\, r\in \mathbb{C}^{n},\; L u=r.\] 
By definition we have
$\overline{u_j}=(\overline{u})_j$, $1\leq j\leq m$. It follows
from the above arguments that $\overline{f}\in
\left(H^1(0,1)\right)^m$, and one can see that 
\[L\overline{u}=\overline{L u}=\overline{r}.\] 
Hence, $\overline{U}\in \mcV$. Moreover, the first two sums of $\ea({\rm Re}\,{U},{\rm Im}\,{U})$ are sums of $m$ integrals. Recall that all the weights are real-valued, nonnegative functions. Since all the integrated functions are real-valued, and the third sum is the sum of real numbers, it follows that ${\ea}({\rm Re}\,{U},{\rm Im}\,{U})\in\mathbb R$. Thus, the first criterion has been checked.

Moreover, if $U$ is a real-valued, then 
$\vert u_j\vert =\vert u\vert_j$, $1\leq j\leq m$, $\vert r_i\vert =\vert r\vert_i$, $1\leq i\leq n$, and one sees as above
that $\vert U\vert\in \mcV$. In particular, 
$\vert\vert u\vert'\vert^2=\vert u'\vert^2$, and
\begin{equation}\label{eq:aUU}
\ea(\vert U\vert,\vert U\vert)=\sum_{j=1}^m\int_0^1 \mu_j c_j(x)\vert u_j'(x)\vert^2 \dx+\sum_{j=1}^m\int_0^1 \mu_j p_j(x)\vert u_j(x)\vert^2 \dx-\langle M\vert r \vert,\vert r \vert\rangle.
\end{equation}
Let us investigate the third term. For $r$ real
\[
\langle Mr,r\rangle=\sum_{i=1}^n r_i\sum_{k=1}^n b_{ik}r_k=\sum_{i=1}^n\sum_{k=1}^n b_{ik}r_ir_k,
\]
and
\[
\langle M\vert r \vert,\vert r \vert\rangle=\sum_{i=1}^n\vert r_i\vert\sum_{k=1}^n b_{ik}\vert r_k\vert =\sum_{i=1}^n\sum_{k=1}^n b_{ik}\vert r_k r_i \vert
\]
For $i\neq k$, by Assumption \ref{as:M}.(2) we have
\[b_{ik}\vert r_i r_k \vert \geq b_{ik}r_ir_k.\]
Since
\[b_{ii}\vert r_i r_i \vert=b_{ii}r_ir_i,\quad i=1,\dots ,n.\]
trivially holds, we obtain that $\langle M\vert r\vert,\vert r \vert\rangle\geq \langle Mr,r\rangle$ for all $r$ real, and by \eqref{eq:aUU},
\[\ea(\vert U\vert,\vert U\vert)\leq \ea(U,U).\]

Finally, take $0\leq U=\left(\begin{smallmatrix} u\\r\end{smallmatrix}\right)\in \mcV$. Then
\[
1\wedge u=1\wedge \left(\begin{smallmatrix}
u_1\\
\vdots\\
u_m\\
\end{smallmatrix}\right)
= \left(\begin{smallmatrix}
1\wedge u_1\\
\vdots\\
1\wedge u_m\\
\end{smallmatrix}\right),
\]
with all the functions $1\wedge u_j\in H^1(0,1)$, hence $1\wedge u\in \left(H^1(0,1)\right)^m$. Clearly,
\[L(1\wedge u)=1\wedge L u=1\wedge r,\]
i.e., $1\wedge U\in \mcV$. Furthermore,
\begin{align*}
\ea(1\wedge U,(U-1)^+)&=\sum_{j=1}^m\int_0^1 \mu_j c_j (1\wedge u_j)'(x) ((u_j-1)^+)'(x) dx\\
&+\sum_{j=1}^m\int_0^1 \mu_j p_j (1\wedge u_j)(x) (u_j-1)^+(x) dx-\langle M(1\wedge r),(r-1)^+\rangle\\[0.3em]
&= \sum_{j=1}^m\int_0^1 \mu_j c_j u_j'(x) {\mathbf{1}}_{\{u_j<1\}}(x) u_j'(x) {\mathbf{1}}_{\{u_j>1\}}(x) dx\\
&+\sum_{j=1}^m\int_0^1 \mu_j p_j u_j(x) {\mathbf{1}}_{\{u_j<1\}}(x) u_j(x) {\mathbf{1}}_{\{u_j>1\}}(x) dx\\
&-\langle M(1\wedge r),(r-1)^+\rangle\\[0.3em]
&=0-\langle M(1\wedge r),(r-1)^+\rangle.
\end{align*}
Using Assumption \ref{as:M} and \cite[Ex.~4.44(3)]{Mu14}, we have that $M$ generates a sub-Markovian semigroup, hence, by \cite[Cor.~2.17]{Ou05}, for $r\geq 0,$
\[-\langle M(1\wedge r),(r-1)^+\rangle\geq 0\]
holds. We have checked also the third criterion, thus the claim follows.
\end{proof}

%For the next lemma we will use the following notations.
%
%\begin{equation}\label{eq:formnorm}
%\Vert U\Vert_{\ea}^2:= \ea(U,U)+\Vert U\Vert^2_{\mathcal{E}_2},\quad U\in\mcV,
%\end{equation}
%and
%\begin{equation}\label{eq:Vnorm}
%\Vert U\Vert_{\mathcal{V}}^2:= \sum_{j=1}^m \Vert u_j\Vert_{H^1(0,1;\mu_j dx)}^2+\Vert r\Vert_{\ell^2}^2,\quad U=\left(\begin{smallmatrix} u\\r\end{smallmatrix}\right)\in \mcV
%\end{equation}
%
%\begin{rem}\label{rem:Vnorm}
%A straightforward computation shows that \eqref{eq:Vnorm} and \eqref{eq:formnorm} define equivalent norms on $\mcV$.
%\end{rem}
%
%\begin{lemma}\label{lem:ultracontr}
%The semigroup $(\mcT_2(t))_{t\geq 0}$ on $\mcE_2$ associated with $\ea$ is ultracontractive. In particular, it satisfies the estimate
%\begin{equation}\label{eq:ultra}
%\Vert \mcT_2(t)U\Vert_{\mcE_\infty} \leq  K t^{-\frac{1}{4}}\Vert U\Vert_{\mcE_2}
%\qquad\hbox{ for all }t\in (0,1],\; U\in{\mcE_2},
%\end{equation}
%for some constant $K$.
%\end{lemma}
%
%\begin{proof}
%We use Remark \ref{rem:Vnorm} and \cite[Thm.~6.3 and following remark]{Ou05} and obtain that it suffices to show that for all $U\in \mcV$ and for some constant $K$
%\begin{equation*}
%\Vert U\Vert_{\mcE_2}\leq K \Vert U\Vert_{\mcV}^{\frac{1}{3}} \cdot \Vert U\Vert^{\frac{2}{3}}_{\mcE_1}
%\end{equation*}
%holds. From now, the proof is a direct analogue of \cite[Thm.~5.4]{MR07}.
%\end{proof}
\textbf{Acknowledgement.} We thank the referee for the constructive criticism and suggestions which helped to improve the quality of the manuscript significantly.

\end{document}